\documentclass[10pt,english]{paper}
\usepackage[T1]{fontenc}
\usepackage[utf8]{luainputenc}
\usepackage{geometry}
\usepackage{amsthm}
\usepackage{amsmath}
\usepackage{amssymb}
\usepackage{esint}

\makeatletter
\theoremstyle{plain}
\newtheorem{thm}{\protect\theoremname}
  \theoremstyle{definition}
  \newtheorem{defn}[thm]{\protect\definitionname}
  \theoremstyle{remark}
  \newtheorem{rem}[thm]{\protect\remarkname}
  \theoremstyle{plain}
  \newtheorem{lem}[thm]{\protect\lemmaname}
  \theoremstyle{plain}
  \newtheorem{prop}[thm]{\protect\propositionname}
  \theoremstyle{plain}
  \newtheorem{cor}[thm]{\protect\corollaryname}

\makeatother

\usepackage{babel}
  \providecommand{\corollaryname}{Corollary}
  \providecommand{\definitionname}{Definition}
  \providecommand{\lemmaname}{Lemma}
  \providecommand{\propositionname}{Proposition}
  \providecommand{\remarkname}{Remark}
\providecommand{\theoremname}{Theorem}

\begin{document}

\title{\qquad\qquad A utility maximization problem 
\\ with state constraint and
non-concave technology}

\author{$\qquad\qquad\qquad\qquad\qquad\,\ \ \ \ \ $\textrm{\textup{Francesco Bartaloni}}\\
$\qquad\qquad\qquad\qquad\qquad\qquad\ \ \ \ $\textrm{\textup{\normalsize{University
of Pisa}}}\textrm{\textup{}}\\
$\qquad\qquad\qquad\qquad\qquad\ \ \ \ \ \ $\textrm{\textup{\normalsize{bartaloni@mail.dm.unipi.it}}}}
\maketitle
\begin{abstract}
We consider an optimal control problem arising in the context of economic
theory of growth, on the lines of the works by Skiba (1978) and Askenazy
- Le Van (1999).

The economic framework of the model is intertemporal infinite horizon utility
maximization. The dynamics involves a state variable representing
total endowment of the social planner or average capital of the representative
dynasty. From the mathematical viewpoint, the main features of the
model are the following: (i) the dynamics is an increasing, unbounded
and not globally concave function of the state; (ii) the state variable
is subject to a static constraint; (iii) the admissible controls are
merely locally integrable in the right half-line. Such assumptions
seem to be weaker than those appearing in most of the existing literature.

We give a direct proof of the existence of an optimal control for
any initial capital $k_{0}\geq0$ and we carry on a qualitative study
of the value function; moreover, using dynamic programming methods,
we show that the value function is a continuous viscosity solution
of the associated Hamilton-Jacobi-Bellman equation.\end{abstract}
\tableofcontents{}

\section{Introduction}

Utility maximization problems represent a fundamental part of modern
economic growth models, since the works by Ramsey (1928), Romer (1986),
Lucas (1988), Barro and Sala-i-Martin (1999).

These models aim to formalize the dynamics of an economy throughout
the quantitative description of the consumers' behaviour. Consumers
are seen as homogeneous entities, as far as their operative decisions
are concerned; hence the time series of their consuming choices, or
consumption path, is represented by a single function, and they as
a collective are named after \emph{social planner}, or simply\emph{
agent}.

The agent's purpose is to maximize the utility in function of the
series of the consumption choices in a fixed time interval; this can
be finite or more often (as far as economic growth literature is concerned)
infinite.

From the application viewpoint, the target of the analysis is the
study of the optimal – in relation to this utility functional – trajectories:
regularity, monotonicity, asymptotic behaviour properties and similar
are expected to be investigated. Hence good existence results are
specially needed, as well as handy sufficient and necessary conditions
for the optimum.

These problems are treated mathematically as optimal control problems;
often external reasons such as the pursuit of more empirical description
power imply the presence of additional state constraints, which we
may call “static” since they do not involve the derivative of the
state variable.

It is worth noticing that the introduction of the static state constraints
usually makes the problem quite harder (and it is sometimes considered
extraneous to the usual setting of control theory). As an example,
we see that the main properties of optimal trajectories are still
not characterized in recent literature, at least in the case of non-concave
production function.

Hence this kind of program is quite complex, especially in the latter
case – and has to be dealt with in many phases. Here we undertake
the work providing an existence result and various necessary conditions
related to the Hamilton-Jacobi-Bellman problem (HJB), remembering
Skiba (1978) and Askenazy - Le Van (1999) and developing part of the
studies carried on by F. Gozzi and D. Fiaschi (2009).

Some technical difficulties arise as an effect of the generality of
the hypothesis on the data, which are supposed to be the reason of
the versatility and wide-range applicability of this model.

First, the dynamics contains a convex-concave function representing
production. It is well known that the presence of non-concavity in
an optimization problem can lead to many difficulties in establishing
the necessary and sufficient conditions for the optimum, as well as
in examining the regularity properties of the value function.

Secondly, the above mentioned presence of the static state constraint
makes any admissibility proof much more complicated than usual.

As a third relevant feature, we require that the admissible controls
are not more than locally integrable in the positive half-line: this
is the maximal class if one wants the control strategy to be a function
and the state equation to have solution. This is a weak regularity
requirement which is of very little help; conversely it generates
unexpected issues in various respects.

We can summarize the main criticalities entailed by these three traits
as follows:

1. Certain questions appear that in other \textquotedbl{}bounded-control\textquotedbl{}
models are not even present, such as the finiteness of the value function
and the well-posedness of the Hamiltonian problem, i.e. the question
whether the value function is a viscosity solution to the HJB equation.
The notion of viscosity solution can be characterized both in terms
of super- and sub-differentials and of test functions; in any case
these auxiliary tools must match the necessary restrictions to the
domain of the Hamiltonian function, at least for the solutions we
are interested in verifying. Fortunately, we are able to prove certain
regularity properties of the value function ensuring that this is
the case.

2. The problem of the existence of an optimal control strategy (for
every fixed initial state) lacks of couplings in the traditional literature
such as Cesari, Zabczyk, Yong-Zhou. It is a natural idea to make use
of the traditional compactness results, in order to generate a convergent
approximation procedure. As we commit ourselves to deal with merely
(locally) integrable control functions, the application of such compactness
results is not straightforward. Indeed, a very careful preliminary
work is needed, providing a uniform localization lemma. The procedure
has then to be further refined so that we can find a limit function
which is admissible in the sense that it satisfies the static state
constraint, and whose functional is (not less then) the limit of the
approximating sequence functionals.

3. Additional work to the usual proof of the fact that the value function
indeed solves HJB is needed; in fact we use the uniform localization
lemma which appears in the optimality construction: the fundamental
Lemma $\ref{controlli limitati e meglio}$.

4. The regularity property stated in Theorem $\ref{Further prop V}$.ii),
which is necessary in order that the HJB problem is well-posed, not
only requires optimal controls. It can be proven by a standard argument
under the hypothesis that the admissible controls are locally bounded;
in our case it shows again to be useful to come back to the preliminary
tools (Lemmas $\ref{controlli limitati e meglio}$ and $\ref{lemma per monotonia V}$)
in order to move around the obstacle and have the result proven with
merely integrable control functions.\\

The contents are consequently arranged: first, the reader will come
across an introductory paragraph which intends to clear up the genesis
of the model and the economic motivations for the assumptions.

Then comes a section dedicated to the preliminary results that are
crucial for the development of the theory. 

Afterwards, some basic properties of the value function are proven,
such as its behaviour near the origin and near $+\infty$. These results
require careful manipulations of the data and some standard results
about ordinary differential equations, but do not require the existence
of optimal control functions.

Next comes the section in which we prove the existence of an optimal
control strategy for every initial state. Here we make wide use of
the preliminary lemmas in association with a special diagonal procedure
generating a weakly convergent sequence of controls from a family
of sequences which, unlike in Ascoli-Arzelà's theorem, are not extracted
neatly one from the other.

After providing the existence theorem, we are able to prove other
important regularity properties of the value function (such as the
Lipschitz-continuity in the closed intervals of $\left(0,+\infty\right)$),
using optimal controls. 

Eventually we give an application of the methods of Dynamic Programming
to our model. As mentioned before, the proof of the admissibility
of the value function as a viscosity solution of HJB is made more
complicated by the use of the preliminary lemmas, but it allows to
obtain the result independently of the regularity of the Hamiltonian
function, which contributes to make this problem peculiar and hopefully
a source of further motives of scientific interest.\\
\\
We note that the admissible controls, modelling the agent's consumptions,
are supposed to be locally integrable also for representativeness
purposes. Since we are able to reach a local boundedness result (the
above mentioned Lemma $\ref{controlli limitati e meglio}$), we could
also have developed most of the optimum existence proof in $L^{2}$,
and then come back to $L^{1}$. Since this space is our natural environment
we have chosen to use the fact that, for a finite-measure space $E$,
a sequence which is uniformly bounded in $L^{\infty}\left(E\right)$
admits a subsequence which is weakly convergent in $L^{1}\left(E\right)$
(this is an easy consequence of the Dunford-Pettis theorem).

\section{The model}

\subsection{Qualitative description}

We assume the existence of a representative dynasty in which all members
share the same endowments and consume the same amount of a certain
good. Our goal is to describe the dynamics of the capital accumulated
by each member of the dynasty in an infinite-horizon period and to
maximize its intertemporal utility (considered as a function of the
quantity of good $c$ that has been consumed). Clearly, consuming
is seen as the agent's control strategy, and the set of consumption
functions (over time) will be a superset of the set of the admissible
control strategies.

First, we need a notion of instantaneous utility, depending on the
consumptions, in order to define the inter-temporal utility functional.
We will assume that instantaneous utility, which we denote by $u$,
is a strictly increasing and strictly concave function of the consumptions,
and that it is twice continuously differentiable. Moreover, we will
assume the usual Inada's conditions, that is to say:
\[
\lim_{c\to0^{+}}u'\left(c\right)=+\infty,\ \lim_{c\to+\infty}u'\left(c\right)=0.
\]

We will also use the following assumptions on $u$:
\[
u\left(0\right)=0,\ \lim_{c\to+\infty}u\left(c\right)=+\infty.
\]
With this material, we can define the inter-temporal utility functional,
which, as usual, must include a (exponential) discount factor expressing
time preference for consumption:
\begin{equation}
U\left(c\left(\cdot\right)\right):=\int_{0}^{+\infty}e^{-\hat{\rho}t}e^{nt}u\left(c\left(t\right)\right)\mbox{d}t\label{Intertemporal utility}
\end{equation}

where $\hat{\rho}\in\mathbb{R}$ is the rate of time preference and
$n\in\mathbb{R}$ is the growth rate of population. The number of
members of the dynasty at time zero is normalized to $1$.

\subsection{Production function and constraints}

We consider the production or output, denoted by $F$, as a function
of the average capital of the representative dynasty, which we denote
by $k$. First, we assume the usual hypothesis of monotonicity, regularity
and unboundedness about the production, that is to say: $F$ is strictly
increasing and continuously differentiable from $\mathbb{R}$ to $\mathbb{R}$,
and
\begin{align*}
 & F\left(0\right)=0,\ \lim_{k\to+\infty}F\left(k\right)=+\infty
\end{align*}

where we may assume $F\left(x\right)<0$ for every $x\in\left(-\infty,0\right)$,
since the assumption that $F$ is defined in $\left(-\infty,0\right)$
is merely technical, as we will see later; this way we distinguish
the “admissible” values of the production function from the ones which
are not.

Next, we make some specific requirements. As we want to deal with
a non-monotonic marginal product of capital, we assume that, in $\left[0,+\infty\right)$,
$F$ is first strictly concave, then strictly convex and then again
strictly concave up to $+\infty$. This means that in the first phase
of capital accumulation, the production shows decreasing returns to
scale, which become increasing from a certain level of \emph{pro capite}
capital $\underline{k}$. Then, when \emph{pro capite} endowment exceed
a threshold $\overline{k}>\underline{k}$, decreasing returns to scale
characterize the production anew.

Moreover, we ask that the marginal product in $+\infty$ is strictly
positive, so that we can deal with endogenous growth. Observe that
this limit surely exists, as $F'$ is (strictly) decreasing in a neighbourhood
of $+\infty$. Of course the assumption is equivalent to the fact
that the average product of capital tends to a strictly positive quantity
for large values of the average stock of capital. Moreover, requiring
that the marginal product has a strictly positive lower bound is necessary
to ensure a positive long-run growth rate.

As far as the agent's behaviour is concerned, the following constraints
must be satisfied, for every time $t\geq0$:
\begin{align*}
 & k\left(t\right)\geq0,\ c\left(t\right)\geq0\\
 & i\left(t\right)+c\left(t\right)\leq F\left(k\left(t\right)\right),\ \dot{k}\left(t\right)=i\left(t\right)
\end{align*}
where $i\left(t\right)$ is the per capita investment at time $t$.
Observe that the first assumption is needed in order to make the agent's
optimal strategy possibly different from the case of monotonic marginal
product. In fact if condition $\forall t\geq0:k\left(t\right)\geq0$
was not present, then heuristically the convex range of production
function would be not relevant to establish the long-run behaviour
of economy, since every agent would have the possibility to get an
amount of resources such that he can fully exploit the increasing
return; therefore only the form of production function for large $k$
would be relevant.

Another heuristic remark turns out to be crucial: the monotonicity
of $u$ respect to $c$ implies that, if $c$ is an optimal consumption
path, then the production is completely allocated between investment
and consumption, that is to say $i\left(t\right)+c\left(t\right)=F\left(k\left(t\right)\right)$
for every $t\geq0$. This remark, combined with the last of the above
conditions implies that the dynamics of capital allocation, for an
initial endowment $k_{0}\geq0$, is described by the following Cauchy's
problem:
\begin{equation}
\begin{cases}
\dot{k}\left(t\right)=F\left(k\left(t\right)\right)-c\left(t\right) & \mbox{ for }t\geq0\\
k\left(0\right)=k_{0}
\end{cases}\label{equazione stato intro}
\end{equation}
Considering the first two constraints, the agent's target can be expressed
the following way: given an initial endowment of capital $k_{0}\geq0$,
maximize the functional in $\eqref{Intertemporal utility}$, when
$c\left(\cdot\right)$ varies among measurable functions which are
everywhere positive in $\left[0,+\infty\right)$ and such that the
unique solution to problem $\eqref{equazione stato intro}$ is also
everywhere positive in $\left[0,+\infty\right)$; the latter requirement
is usually called a \emph{state constraint}.

A few reflections are still necessary in order to begin the analytic
work. First, we will consider only the case when the time discount
rate $\hat{\rho}$ and the population growth rate $n$ satisfy
\[
\hat{\rho}-n>0,
\]
which is the most interesting from the economic point of view. Second,
we weaken the requirement that $c$ is measurable and positive in
$\left[0,+\infty\right)$ (in order that $c$ is admissible) to the
requirement that $c$ is locally integrable and almost everywhere
positive in $\left[0,+\infty\right)$.

Finally, we need another assumption about instantaneous utility $u$
so that the functional in $\eqref{Intertemporal utility}$ is finite.
To identify the best hypothesis, we temporarily restrict our attention
to the particular but significant case in which $u$ is a concave
power function and $F$ is linear; namely:
\begin{align*}
u\left(c\right)=c^{1-\sigma}, & \quad c\geq0\\
F\left(k\right)=Lk, & \quad k\geq0
\end{align*}
for some $\sigma\in\left(0,1\right)$ and $L>0$ (of course in this
case $F$ does not satisfy all of the previous assumptions). Using
Gronwall's Lemma, it is easy to verify that for any admissible control
$c$ (starting from an initial state $k_{0}$) and for every time
$t\geq0$, $\int_{0}^{t}c\left(s\right)\mbox{d}s\leq k_{0}e^{Lt}$.
Hence, setting $\rho=\hat{\rho}-n$:
\begin{eqnarray*}
U\left(c\left(\cdot\right)\right) & = & \lim_{T\to+\infty}\int_{0}^{T}e^{-\rho t}u\left(c\left(t\right)\right)\mbox{d}t\\
 & = & \lim_{T\to+\infty}e^{-\rho T}\int_{0}^{T}u\left(c\left(s\right)\right)\mbox{d}s+\lim_{T\to+\infty}\rho\int_{0}^{T}e^{-\rho t}\int_{0}^{t}u\left(c\left(s\right)\right)\mbox{d}s\mbox{d}t.
\end{eqnarray*}
Hence using Jensen inequality, we reduce the problem of the convergence
of $U\left(c\left(\cdot\right)\right)$ to the problem of the convergence
of
\[
\int_{1}^{+\infty}te^{-\rho t}e^{L\left(1-\sigma\right)t}\mbox{d}t
\]
which is equivalent to the condition $L\left(1-\sigma\right)<\rho$.
Perturbing this clause by the addition of a positive quantity $\epsilon_{0}$
we get $\left(L+\epsilon_{0}\right)\left(1-\sigma\right)<\rho-\epsilon_{0}$
which is in its turn equivalent to the requirement that the function
$e^{\epsilon_{0}t}e^{-\rho t}\left(e^{\left(L+\epsilon_{0}\right)t}\right)^{1-\sigma}=e^{\epsilon_{0}t}e^{-\rho t}u\left(e^{\left(L+\epsilon_{0}\right)t}\right)$
tends to $0$ as $t\to+\infty$.

Turning back to the general case, we are suggested to assume precisely
the same condition, taking care of defining the constant $L$ as $\lim_{k\to+\infty}F'\left(k\right)$
(which has already been assumed to be strictly positive).

\subsection{Quantitative description}

Hence the mathematical frame of the economic problem can be defined
precisely as follows:
\begin{defn}
For every $k_{0}\geq0$ and for every $c\in\mathcal{L}_{loc}^{1}\left(\left[0,+\infty\right),\mathbb{R}\right)$: 

$k\left(\cdot;k_{0},c\right)$ is the only solution to the Cauchy's
problem
\begin{equation}
\begin{cases}
k\left(0\right)=k_{0}\\
\dot{k}\left(t\right)=F\left(k\left(t\right)\right)-c\left(t\right) & \ t\geq0
\end{cases}\label{eq di stato}
\end{equation}
in the unknown $k$, where $F:\mathbb{R}\to\mathbb{R}$ has the following
properties:
\begin{align*}
 & F\in\mathcal{C}^{1}\left(\mathbb{R},\mathbb{R}\right),\, F'>0\mbox{ in }\mathbb{R},\, F\left(0\right)=0,\,\lim_{x\to+\infty}F\left(x\right)=+\infty,\ {\displaystyle \lim_{x\to+\infty}F'\left(x\right)>0},\\
 & F\mbox{ is concave in }\left[0,\underline{k}\right]\cup\left[\overline{k},+\infty\right)\mbox{ for some }0<\underline{k}<\overline{k}\mbox{ and }F\mbox{ is convex over }\left[\underline{k},\overline{k}\right]
\end{align*}
Moreover, we set $L:={\displaystyle \lim_{x\to+\infty}F'}\left(x\right)$.
\end{defn}
$\,$
\begin{defn}
Let $k_{0}\geq0$ .

The set of \emph{admissible consumption strategies} with initial capital
$k_{0}$ is 
\[
\Lambda\left(k_{0}\right):=\left\{ c\in\mathcal{L}_{loc}^{1}\left(\left[0,+\infty\right),\mathbb{R}\right)/c\geq0\mbox{ almost everywhere},\, k\left(\cdot;k_{0},c\right)\geq0\right\} 
\]
The \emph{intertemporal utility functional} $U\left(\cdot;k_{0}\right)$:$\Lambda\left(k_{0}\right)\to\mathbb{R}$
is
\[
U\left(c;k_{0}\right):=\int_{0}^{+\infty}e^{-\rho t}u\left(c\left(t\right)\right)\mbox{d}t\ \,\forall c\in\Lambda\left(k_{0}\right)
\]
where $\rho>0$, and the function $u:\left[0,+\infty\right)\to\mathbb{R}$,
representing instantaneous utility, is strictly increasing and strictly
concave and satisfies:
\begin{align}
 & u\in\mathcal{C}^{2}\left(\left(0,+\infty\right),\mathbb{R}\right)\cap\mathcal{C}^{0}\left(\left[0,+\infty\right),\mathbb{R}\right),\, u\left(0\right)=0,\,\lim_{x\to+\infty}u\left(x\right)=+\infty\nonumber \\
 & {\displaystyle \lim_{x\to0^{+}}u'\left(x\right)=+\infty,\ \lim_{x\to+\infty}u'\left(x\right)=0}\label{assumption u}\\
 & {\displaystyle \exists\epsilon_{0}>0:\lim_{t\to+\infty}e^{\epsilon_{0}t}e^{-\rho t}u\left(e^{\left(L+\epsilon_{0}\right)t}\right)=0}\nonumber 
\end{align}
The \emph{value function} $V:\left[0,+\infty\right)\to\mathbb{R}$
is
\[
V\left(k_{0}\right):=\sup_{c\in\Lambda\left(k_{0}\right)}U\left(c;k_{0}\right)\ \forall k_{0}\geq0
\]
\end{defn}
\begin{rem}
\label{remark 1}The last condition in $\eqref{assumption u}$ implies:
\begin{align*}
 & \int_{0}^{+\infty}e^{-\rho t}u\left(e^{\left(L+\epsilon_{0}\right)t}\right)\mbox{d}t<+\infty,\ \int_{0}^{+\infty}te^{-\rho t}u\left(e^{\left(L+\epsilon_{0}\right)t}\right)\mbox{d}t<+\infty.
\end{align*}

\end{rem}

\section{Preliminary results}
\begin{rem}
\label{F lip vero con esistenza}Set
\[
\overline{M}:=\max_{\left[0,+\infty\right)}F'=\max\left\{ F'\left(0\right),F'\left(\bar{k}\right)\right\} .
\]
Recalling that $F$ is strictly increasing with $F\left(0\right)=0$,
we see that, for any $x,y\in\left[0,+\infty\right)$:
\begin{align*}
 & \left|F\left(x\right)-F\left(y\right)\right|\leq\overline{M}\left|x-y\right|\\
 & F\left(x\right)\leq\overline{M}x
\end{align*}
In particular $F$ is Lipschitz-continuous.

This implies that the Cauchy's problem $\eqref{eq di stato}$ admits
a unique global solution (that is to say, defined on $\left[0,+\infty\right)$).

Indeed the mapping
\[
\mathcal{F}\left(k\right)\left(t\right):=k_{0}+\int_{0}^{t}F\left(k\left(s\right)\right)\mbox{d}s-\int_{0}^{t}c\left(s\right)\mbox{d}s
\]
 is a contraction on the space $X:=\left(\mathcal{C}^{0}\left(\left[0,\frac{1}{1+\overline{M}}\right]\right),\left\Vert \cdot\right\Vert _{\infty}\right)$,
and so admits a unique fixed point $k\left(\cdot;k_{0},c\right)$.
Considering the mapping
\[
\mathcal{F}\left(k\right)\left(t\right):=k\left(\frac{1}{1+\overline{M}};k_{0},c\right)+\int_{\frac{1}{1+\overline{M}}}^{t}F\left(k\left(s\right)\right)\mbox{d}s-\int_{\frac{1}{1+\overline{M}}}^{t}c\left(s\right)\mbox{d}s
\]
on the space $X':=\left(\mathcal{C}^{0}\left(\left[\frac{1}{1+\overline{M}},\frac{2}{1+\overline{M}}\right]\right),\left\Vert \cdot\right\Vert _{\infty}\right)$,
one can extend $k\left(\cdot;k_{0},c\right)$ to the interval $\left[\frac{1}{1+\overline{M}},\frac{2}{1+\overline{M}}\right]$,
and so on.
\end{rem}
$\,$
\begin{rem}
\label{Lemma Comp ODE}We recall that if $k_{1}$ and $k_{2}$ are
two solutions of $\eqref{eq di stato}$, then the function
\[
h\left(t\right):=\begin{cases}
{\displaystyle \frac{F\left(k_{1}\left(t\right)\right)-F\left(k_{2}\left(t\right)\right)}{k_{1}\left(t\right)-k_{2}\left(t\right)}} & \mbox{ if }k_{1}\left(t\right)\neq k_{2}\left(t\right)\\
F'\left(k_{1}\left(t\right)\right) & \mbox{ if }k_{1}\left(t\right)=k_{2}\left(t\right)
\end{cases}
\]
is continuous in $\left[0,+\infty\right)$.

As a consequence, we have a well known comparison result, which in
our case can be stated as follows:
\end{rem}
\emph{Let $k_{1},k_{2}\geq0$, $c_{1},c_{2}\in\mathcal{L}_{loc}^{1}\left(\left[0,+\infty\right),\mathbb{R}\right)$,
$T_{0}\geq0$ and $T_{1}\in\left(T_{0},+\infty\right]$ such that
$c_{1}\leq c_{2}$ almost everywhere in $\left[T_{0},T_{1}\right]$.
Then the following implications hold:}
\begin{eqnarray}
k\left(T_{0};k_{1},c_{1}\right)=k\left(T_{0};k_{2},c_{2}\right) & \implies & \forall t\in\left[T_{0},T_{1}\right]:k\left(t;k_{1},c_{1}\right)\geq k\left(t;k_{2},c_{2}\right)\label{eq: comp ODE weak}\\
k\left(T_{0};k_{1},c_{1}\right)>k\left(T_{0};k_{2},c_{2}\right) & \implies & \forall t\in\left[T_{0},T_{1}\right]:k\left(t;k_{1},c_{1}\right)>k\left(t;k_{2},c_{2}\right).\label{eq: comp ODE strong}
\end{eqnarray}

$\,$
\begin{lem}
\label{minorante convessa}There exists a function $g:\left(0,+\infty\right)\to\left(0,+\infty\right)$
which is convex, decreasing and such that
\[
g\left(x\right)\leq u'\left(x\right)\quad\forall x>0.
\]
\end{lem}
\begin{proof}
Let
\begin{align*}
 & \Sigma_{u'}:=\left\{ \left(x,y\right)\in\left(0,+\infty\right)^{2}/y\geq u'\left(x\right)\right\} \\
 & K_{u'}:=\bigcap\left\{ K\in\mathcal{P}\left(\mathbb{R}^{2}\right)/K=\overline{K},\, K\mbox{ is convex},\, K\supseteq\Sigma_{u'}\right\} .
\end{align*}
In particular $K_{u'}$ is a closed-convex superset of $\Sigma_{u'}$.
Observe that, for any $x>0$, the function $H_{x}\left(y\right):=\left(x,y\right)$
belongs to $\mathcal{C}^{0}\left(\mathbb{R},\mathbb{R}^{2}\right)$,
so any set of the form 
\[
\left\{ y\geq0/\left(x,y\right)\in K_{u'}\right\} =H_{x}^{-1}\left(K_{u'}\right)\bigcap\left[0,+\infty\right)
\]
is closed in $\mathbb{R}$, and consequently it has a minimum element.
Now define
\[
\forall x>0:g\left(x\right):=\min\left\{ y\geq0/\left(x,y\right)\in K_{u'}\right\} .
\]
i) This definition implies that for every $\left(x,y\right)\in K_{u'}$,
$g\left(x\right)\leq y$; hence 
\[
g\left(x\right)\leq u'\left(x\right)\quad\forall x>0
\]
because for any $x>0$, $\left(x,u'\left(x\right)\right)\in\Sigma_{u'}\subseteq K_{u'}$.

ii) In the second place, $g$ is convex in $\left(0,+\infty\right)$.
Let $x_{0},x_{1}>0$ and $\lambda\in\left(0,1\right)$. By definition
of $g$, $\left(x_{0},g\left(x_{0}\right)\right),\left(x_{1},g\left(x_{1}\right)\right)\in K_{u'}$,
which is a convex set. Hence
\[
\left(1-\lambda\right)\left(x_{0},g\left(x_{0}\right)\right)+\lambda\left(x_{1},g\left(x_{1}\right)\right)\in K_{u'}.
\]
By the first property in i), this implies
\[
g\left(\left(1-\lambda\right)x_{0}+\lambda x_{1}\right)\leq\left(1-\lambda\right)g\left(x_{0}\right)+\lambda g\left(x_{1}\right).
\]
iii) $g$ is decreasing. Indeed, take $0<x_{0}<x_{1}$. By ii) and
by definition of convexity, for every $n\in\mathbb{N}$:
\begin{eqnarray*}
g\left(n\left(x_{1}-x_{0}\right)+x_{0}\right) & \geq & n\left[g\left(x_{1}\right)-g\left(x_{0}\right)\right]+g\left(x_{0}\right).
\end{eqnarray*}
Hence by the assumptions on $u$ and by i):
\begin{eqnarray*}
0 & = & \lim_{n\to+\infty}u'\left(n\left(x_{1}-x_{0}\right)+x_{0}\right)\geq\limsup_{n\to+\infty}g\left(n\left(x_{1}-x_{0}\right)+x_{0}\right)\\
 & \geq & \lim_{n\to+\infty}n\left[g\left(x_{1}\right)-g\left(x_{0}\right)\right]+g\left(x_{0}\right)
\end{eqnarray*}
which implies $g\left(x_{1}\right)\leq g\left(x_{0}\right)$.

iv) Observe that the definition of $g$ does not exclude that $g\left(x\right)=0$
for some $x>0$. Indeed we show that $g>0$ in $\left(0,+\infty\right)$.

Fix $x>0$, and consider the closed-convex aproximation of $\Sigma_{u'}$
\[
K_{x}:=\left\{ \left(t,y\right)\in\left[0,x\right]\times\left[0,+\infty\right)/y\geq\frac{\min_{\left[0,x\right]}u'}{x}\left(x-t\right)\right\} \bigcup\left[x,+\infty\right)\times\left[0,+\infty\right).
\]
By construction $K_{u'}\subseteq K_{x}$ which implies $\left(t,g\left(t\right)\right)\in K_{x}$
for any $t>0$. In particular, for every $t\in\left(0,x\right)$:
\[
g\left(t\right)\geq\frac{\min_{\left[0,x\right]}u'}{x}\left(x-t\right)>0
\]
because $u'>0$. This is precisely the fact that allows us to repeat
this construction for every $x>0$, which ensures that $g>0$ in $\left(0,+\infty\right)$.
\end{proof}
$\,$
\begin{rem}
\label{F Lip} The function $h$ defined in Remark $\ref{Lemma Comp ODE}$
satisfies
\[
\left|h\right|\leq\overline{M}.
\]
where $\overline{M}$ is defined as in Remark $\ref{F lip vero con esistenza}$.
\end{rem}
$\,$
\begin{rem}
\label{prime stime Ic}Let $k_{0}\geq0$ and $c\in\Lambda\left(k_{0}\right)$.
Then, for every $t\geq0$:
\begin{eqnarray*}
k\left(t;k_{0},c\right) & \leq & k_{0}e^{\overline{M}t}\\
\int_{0}^{t}c\left(s\right)\mbox{d}s & \leq & k_{0}e^{\overline{M}t}
\end{eqnarray*}
Indeed, by Remark $\ref{F lip vero con esistenza}$ and remembering
that $c\geq0$, we have, for every $t\geq0$, $\dot{k}\left(t;k_{0},c\right)\leq\overline{M}k\left(t;k_{0},c\right)$
- which implies by $\eqref{eq: comp ODE weak}$:
\[
k\left(t;k_{0},c\right)\leq k_{0}e^{\overline{M}t}\quad\forall t\geq0.
\]

Now integrating both sides of the state equation, again by Remark
$\ref{F lip vero con esistenza}$ and by the fact that $k\left(\cdot;k_{0},c\right)\geq0$
we see that, for every $t\geq0$:
\begin{eqnarray*}
\int_{0}^{t}c\left(s\right)\mbox{d}s & = & k_{0}-k\left(t;k_{0},c\right)+\int_{0}^{t}F\left(k\left(s;k_{0},c\right)\right)\mbox{d}s\\
 & \leq & k_{0}+\overline{M}\int_{0}^{t}k\left(s;k_{0},c\right)\mbox{d}s\\
 & \leq & k_{0}+\overline{M}k_{0}\int_{0}^{t}e^{\overline{M}s}\mbox{d}s=k_{0}e^{\overline{M}t}.
\end{eqnarray*}
\end{rem}
\begin{lem}
\label{controlli limitati e meglio}There exists a function $N:\left(0,+\infty\right)^{2}\to\left(0,+\infty\right)$,
increasing in both variables, such that:

for every $\left(k_{0},T\right)\in\left(0,+\infty\right)^{2}$ and
every $c\in\Lambda\left(k_{0}\right)$, there exists a control function
$c^{T}\in\Lambda\left(k_{0}\right)$ satisfying
\begin{eqnarray*}
 & U\left(c^{T};k_{0}\right)\geq U\left(c;k_{0}\right)\\
 & c^{T}=c\wedge N\left(k_{0},T\right)\mbox{ almost everywhere in }\left[0,T\right]
\end{eqnarray*}
In particular, $c^{T}$ is bounded above, in $\left[0,T\right]$,
by a quantity which does not depend on the original control $c$,
but only on $T$ and on the initial status $k_{0}$.\end{lem}
\begin{proof}
Let $g$ be the function defined in Lemma $\ref{minorante convessa}$
and $\beta:=\frac{\log\left(1+\overline{M}\right)}{\overline{M}}$.
Define, for every $\left(k_{0},T\right)\in\left(0,+\infty\right)^{2}$
:
\begin{align*}
 & \alpha\left(k_{0},T\right):=\beta e^{-\rho\left(T+\beta\right)}g\left[k_{0}\left(\frac{e^{\overline{M}\left(T+\beta\right)}}{\beta}+e^{\overline{M}T}\right)\right]\\
 & N\left(k_{0},T\right):=\inf\left\{ \tilde{N}>0/\forall N\geq\tilde{N}:u'\left(N\right)<\alpha\left(k_{0},T\right)\right\} .
\end{align*}
 In the first place, $N\left(k_{0},T\right)\neq+\infty$, because
$\alpha\left(k_{0},T\right)>0$ for every $k_{0}>0$, $T>0$ and $\lim_{N\to+\infty}u'\left(N\right)=0$.

In the second place, $u'\left(\left(0,+\infty\right)\right)=\left(0,+\infty\right)$,
which implies $N\left(k_{0},T\right)>0$: otherwise, since $\left(u'\right)^{-1}$$\left(\alpha\left(k_{0},T\right)\right)>0$,
there would exist $N>0$ such that 
\begin{align*}
 & N<\left(u'\right)^{-1}\left(\alpha\left(k_{0},T\right)\right)\\
 & u'\left(N\right)<\alpha\left(k_{0},T\right)
\end{align*}
which is absurd because $u'$ is decreasing; hence the quantity $u'\left(N\left(k_{0},T\right)\right)$
is well defined. Moreover by the continuity of $u'$,
\begin{equation}
u'\left(N\left(k_{0},T\right)\right)\leq\alpha\left(k_{0},T\right).\label{N va bene}
\end{equation}
The function $N\left(\cdot,\cdot\right)$ is also increasing in both
variables, because $\alpha\left(\cdot,\cdot\right)$ is decreasing
in both variables and $u'$ is decreasing.

Indeed, for $k_{0}\leq k_{1}$ and for a fixed $T>0$, suppose that
$N\left(k_{1},T\right)<N\left(k_{0},T\right)$. Then by definition
of infimum we could choose $\tilde{N}\in\left[N\left(k_{1},T\right),N\left(k_{0},T\right)\right)$
such that $u'\left(\tilde{N}\right)<\alpha\left(k_{1},T\right)$,
which implies
\[
u'\left(\tilde{N}\right)<\alpha\left(k_{0},T\right)
\]
by the monotonicity of $\alpha$. But since $\tilde{N}>0$, $\tilde{N}<N\left(k_{0},T\right)$
there also exists $N\geq\tilde{N}$ such that $u'\left(N\right)\geq\alpha\left(k_{0},T\right)$
which implies, by the monotonicity of $u'$,
\[
u'\left(\tilde{N}\right)\geq\alpha\left(k_{0},T\right),
\]
a contradiction. With an analogous argument we prove that $N\left(\cdot,\cdot\right)$
is increasing in the second variable.

Now let $k_{0},T>0$ and $c\in\Lambda\left(k_{0}\right)$ as in the
hypothesis. If $c\leq N\left(k_{0},T\right)$ almost everywhere in
$\left[0,T\right]$, then define $c^{T}:=c$. If, on the contrary,
$c>N\left(k_{0},T\right)$ in a non-negligible subset of $\left[0,T\right]$,
then define:
\[
c^{T}\left(t\right):=\begin{cases}
c\left(t\right)\wedge N\left(k_{0},T\right) & \mbox{ if }t\in\left[0,T\right]\\
c\left(t\right)+I_{T} & \mbox{ if }t\in\left(T,T+\beta\right]\\
c\left(t\right) & \mbox{ if }t>T+\beta
\end{cases}
\]
where $I_{T}:=\int_{0}^{T}e^{-\rho t}\left(c\left(t\right)-c\left(t\right)\wedge N\left(k_{0},T\right)\right)\mbox{d}t$.
Observe that by Remark $\ref{prime stime Ic}$: 
\begin{eqnarray}
0<\ I_{T} & \leq & \int_{0}^{T}\left(c\left(t\right)-c\left(t\right)\wedge N\left(k_{0},T\right)\right)\mbox{d}t\nonumber \\
 & \leq & \int_{0}^{T}c\left(t\right)\mbox{d}t\nonumber \\
 & \leq & k_{0}e^{\overline{M}T}\label{stime su I_T,N}
\end{eqnarray}
In order to prove the admissibility of such control function, we compare
the orbit $k:=k\left(\cdot;k_{0},c\right)$ to the orbit $k^{T}:=k\left(\cdot;k_{0},c^{T}\right)$.
In the first place, observe that by $\eqref{eq: comp ODE weak}$ and
by definition of $c^{T}$:
\begin{equation}
k^{T}\left(t\right)\geq k\left(t\right)\quad\forall t\in\left[0,T\right]\label{confronto su [0,T]}
\end{equation}
Now by the state equation, we have:
\begin{equation}
\dot{k^{T}}-\dot{k}=F\left(k^{T}\right)-F\left(k\right)+c-c^{T}.\label{sottr state eq}
\end{equation}
Set for every $t\geq0$:
\[
h\left(t\right):=\begin{cases}
\frac{F\left(k^{T}\left(t\right)\right)-F\left(k\left(t\right)\right)}{k^{T}\left(t\right)-k\left(t\right)} & \mbox{ if }k^{T}\left(t\right)\neq k\left(t\right)\\
F'\left(k\left(t\right)\right) & \mbox{ if }k^{T}\left(t\right)=k\left(t\right)
\end{cases}
\]
Hence by $\eqref{sottr state eq}$
\[
\dot{k^{T}}\left(t\right)-\dot{k}\left(t\right)=h\left(t\right)\left[k^{T}\left(t\right)-k\left(t\right)\right]+c\left(t\right)-c^{T}\left(t\right)\quad\forall t\geq0.
\]
By Remark $\ref{Lemma Comp ODE}$, the function $h$ is continuous
in $\left[0,+\infty\right)$, so this is a typical linear equation
with measurable coefficient of degree one, satisfied by $k^{T}-k$.
Hence, multiplying both sides by the continuous function $t\to\exp\left(-\int_{0}^{t}h\left(s\right)\mbox{d}s\right)$,
we obtain:
\[
\frac{\mbox{d}}{\mbox{d}t}\left\{ \left[k^{T}\left(t\right)-k\left(t\right)\right]e^{-\int_{0}^{t}h\left(s\right)\mbox{d}s}\right\} =\left[c\left(t\right)-c^{T}\left(t\right)\right]e^{-\int_{0}^{t}h\left(s\right)\mbox{d}s}\quad\forall t\geq0
\]
which implies, integrating between $0$ and any $t\geq0$:
\begin{eqnarray}
k^{T}\left(t\right)-k\left(t\right) & = & \int_{0}^{t}\left[c\left(s\right)-c^{T}\left(s\right)\right]e^{\int_{s}^{t}h}\mbox{d}s\label{orbite a confronto}
\end{eqnarray}
Now observe that 
\begin{equation}
h\leq\overline{M}\mbox{ in }\left[0,+\infty\right)\mbox{ and }h\geq0\mbox{ in }\left[0,T\right]\label{fondam aggiunta}
\end{equation}

by $\eqref{confronto su [0,T]}$ and the monotonicity of $F$. Set
$t\in\left(T,T+\beta\right]$; then by $\eqref{orbite a confronto}$
and $\eqref{fondam aggiunta}$:
\begin{eqnarray}
k^{T}\left(t\right)-k\left(t\right) & = & \int_{0}^{T}\left[c\left(s\right)-c\left(s\right)\wedge N\left(k_{0},T\right)\right]e^{\int_{s}^{t}h}\mbox{d}s-I_{T}\cdot\int_{T}^{t}e^{\int_{s}^{t}h}\mbox{d}s\nonumber \\
 & \geq & \int_{0}^{T}\left[c\left(s\right)-c\left(s\right)\wedge N\left(k_{0},T\right)\right]\mbox{d}s-I_{T}\cdot\int_{T}^{t}e^{\overline{M}\left(t-s\right)}\mbox{d}s\nonumber \\
 & \geq & \int_{0}^{T}e^{-\rho s}\left[c\left(s\right)-c\left(s\right)\wedge N\left(k_{0},T\right)\right]\mbox{d}s-I_{T}\cdot\int_{T}^{T+\beta}e^{\overline{M}\left(T+\beta-s\right)}\mbox{d}s\nonumber \\
 & = & I_{T}\left(1-\frac{e^{\overline{M}\beta}-1}{\overline{M}}\right)=0\label{confronto[T,T+beta]}
\end{eqnarray}
This also implies, by $\eqref{eq: comp ODE weak}$ and by definition
of $c^{T}$,
\[
k^{T}\left(t\right)\geq k\left(t\right)\quad\forall t\geq T+\beta
\]

Such inequality, together with $\eqref{confronto su [0,T]}$ and $\eqref{confronto[T,T+beta]}$,
gives us the general inequality
\begin{align*}
 & k^{T}\left(t\right)\geq k\left(t\right)\geq0\quad\forall t\geq0.
\end{align*}
This implies, associated with the obvious fact that $c^{T}\geq0$
almost everywhere in $\left[0,+\infty\right)$, that $c^{T}\in\Lambda\left(k_{0}\right)$.

Now we prove the ``optimality'' property of $c^{T}$ respect to
$c$. By the concavity of $u$, and setting $N:=N\left(k_{0},T\right)$
for simplicity of notation, we have:
\begin{eqnarray}
U\left(c;k_{0}\right)-U\left(c^{T};k_{0}\right) & = & \int_{0}^{+\infty}e^{-\rho t}\left[u\left(c\left(t\right)\right)-u\left(c^{T}\left(t\right)\right)\right]\mbox{d}t\nonumber \\
 & = & \int_{\left[0,T\right]\cap\left\{ c\geq N\right\} }e^{-\rho t}\left[u\left(c\left(t\right)\right)-u\left(c\left(t\right)\wedge N\right)\right]\mbox{d}t\nonumber \\
 &  & +\int_{T}^{T+\beta}e^{-\rho t}\left[u\left(c\left(t\right)\right)-u\left(c\left(t\right)+I_{T}\right)\right]\mbox{d}t\nonumber \\
 & \leq & \int_{\left[0,T\right]\cap\left\{ c\geq N\right\} }e^{-\rho t}u'\left(c\left(t\right)\wedge N\right)\left[c\left(t\right)-c\left(t\right)\wedge N\right]\mbox{d}t\nonumber \\
 &  & -I_{T}\int_{T}^{T+\beta}e^{-\rho t}u'\left(c\left(t\right)+I_{T}\right)\mbox{d}t\nonumber \\
 & = & u'\left(N\right)\int_{0}^{T}e^{-\rho t}\left[c\left(t\right)-c\left(t\right)\wedge N\right]\mbox{d}t\nonumber \\
 &  & -I_{T}\int_{T}^{T+\beta}e^{-\rho t}u'\left(c\left(t\right)+I_{T}\right)\mbox{d}t\nonumber \\
 & = & I_{T}\left[u'\left(N\right)-\int_{T}^{T+\beta}e^{-\rho t}u'\left(c\left(t\right)+I_{T}\right)\mbox{d}t\right]\label{nuovo contr quasi meglio}
\end{eqnarray}
Now we exhibit a certain lower bound wich is independent on the particular
control function $c$. By Jensen inequality, by Lemma $\ref{minorante convessa}$
and by $\eqref{stime su I_T,N}$, we have:
\begin{eqnarray*}
\int_{T}^{T+\beta}e^{-\rho t}u'\left(c\left(t\right)+I_{T}\right)\mbox{d}t & \geq & \int_{T}^{T+\beta}e^{-\rho t}g\left(c\left(t\right)+I_{T}\right)\mbox{d}t\\
 & \geq & e^{-\rho\left(T+\beta\right)}\int_{T}^{T+\beta}g\left(c\left(t\right)+I_{T}\right)\mbox{d}t\\
 & \geq & \beta e^{-\rho\left(T+\beta\right)}g\left(\frac{1}{\beta}\int_{T}^{T+\beta}\left[c\left(t\right)+I_{T}\right]\mbox{d}t\right)\\
 & \geq & \beta e^{-\rho\left(T+\beta\right)}g\left(\frac{1}{\beta}\int_{0}^{T+\beta}c\left(t\right)\mbox{d}t+I_{T}\right)\\
 & \geq & \beta e^{-\rho\left(T+\beta\right)}g\left[k_{0}\left(\frac{e^{\overline{M}\left(T+\beta\right)}}{\beta}+e^{\overline{M}T}\right)\right]\\
 & = & \alpha\left(k_{0},T\right).
\end{eqnarray*}
Hence by $\eqref{N va bene}$ and $\eqref{nuovo contr quasi meglio}$:
\begin{eqnarray*}
U\left(c;k_{0}\right)-U\left(c^{T};k_{0}\right) & \leq & I_{T}\left[u'\left(N\left(k_{0},T\right)\right)-\int_{T}^{T+\beta}e^{-\rho t}u'\left(c\left(t\right)+I_{T}\right)\mbox{d}t\right]\\
 & \leq & I_{T}\left[u'\left(N\left(k_{0},T\right)\right)-\alpha\left(k_{0},T\right)\right]\leq0.
\end{eqnarray*}
\end{proof}
\begin{lem}
\label{lemma per monotonia V}Let $0<k_{0}<k_{1}$ and $c\in\Lambda\left(k_{0}\right)$.
Then there exists a control function $\underline{c}^{k_{1}-k_{0}}\in\Lambda\left(k_{1}\right)$
such that
\[
U\left(\underline{c}^{k_{1}-k_{0}};k_{1}\right)-U\left(c;k_{0}\right)\geq u'\left(N\left(k_{0},k_{1}-k_{0}\right)+1\right)\int_{0}^{k_{1}-k_{0}}e^{-\rho t}\mbox{d}t
\]
where $N$ is the function defined in \emph{Lemma} $\ref{controlli limitati e meglio}$.\end{lem}
\begin{proof}
Fix $k_{0},k_{1}$ and $c$ as in the hypothesis and take $c^{k_{1}-k_{0}}$
as in Lemma $\ref{controlli limitati e meglio}$ (where it is understood
that $T=k_{1}-k_{0}$).Then define:
\[
\underline{c}^{k_{1}-k_{0}}\left(t\right):=\begin{cases}
c^{k_{1}-k_{0}}\left(t\right)+1 & \mbox{ if }t\in\left[0,k_{1}-k_{0}\right)\\
c^{k_{1}-k_{0}}\left(t\right) & \mbox{ if }t\geq k_{1}-k_{0}
\end{cases}
\]
In the first place we prove that $\underline{c}^{k_{1}-k_{0}}\in\Lambda\left(k_{1}\right)$,
showing that 
\begin{equation}
\underline{k}:=k\left(\cdot;k_{1};\underline{c}^{k_{1}-k_{0}}\right)>k\left(\cdot;k_{0},c^{k_{1}-k_{0}}\right)=:k\label{orbite in sec Lemma}
\end{equation}
over $\left(0,+\infty\right)$. Suppose by contradiction that this
is not true, and take $\tau:=\inf\left\{ t>0/\underline{k}\left(t\right)\leq k\left(t\right)\right\} $.
Then by the continuity of the orbits, $\underline{k}\left(\tau\right)\leq k\left(\tau\right)$,
which implies $\tau>0$. Considering the orbits as solutions to an
integral equation we have:
\[
k\left(\tau\right)=k_{0}+\int_{0}^{\tau}F\left(k\left(t\right)\right)\mbox{d}t-\int_{0}^{\tau}c^{k_{1}-k_{0}}\left(t\right)\mbox{d}t
\]
\[
\underline{k}\left(\tau\right)=k_{1}+\int_{0}^{\tau}F\left(\underline{k}\left(t\right)\right)\mbox{d}t-\int_{0}^{\tau}c^{k_{1}-k_{0}}\left(t\right)\mbox{d}t-\min\left\{ \tau,k_{1}-k_{0}\right\} .
\]
Hence
\begin{eqnarray*}
0\geq\underline{k}\left(\tau\right)-k\left(\tau\right) & = & k_{1}-k_{0}+\int_{0}^{\tau}\left[F\left(\underline{k}\left(t\right)\right)-F\left(k\left(t\right)\right)\right]\mbox{d}t-\min\left\{ \tau,k_{1}-k_{0}\right\} \\
 & \geq & \int_{0}^{\tau}\left[F\left(\underline{k}\left(t\right)\right)-F\left(k\left(t\right)\right)\right]\mbox{d}t
\end{eqnarray*}
By the definition of $\tau$ and the strict monotonicity of $F$,
this quantity must be strictly positive, which is absurd. Hence
\begin{align*}
 & k\left(\cdot;k_{1};\underline{c}^{k_{1}-k_{0}}\right)>k\left(\cdot;k_{0},c^{k_{1}-k_{0}}\right)\geq0\mbox{ in }\left[0,+\infty\right)\\
 & \underline{c}^{k_{1}-k_{0}}\geq c^{k_{1}-k_{0}}\geq0\mbox{ a.e. in }\left[0,+\infty\right)
\end{align*}
which implies $\underline{c}^{k_{1}-k_{0}}\in\Lambda\left(k_{0}\right)$.

In the second place, remembering the properties of $c^{k_{1}-k_{0}}$
given by Lemma $\ref{controlli limitati e meglio}$, we have 
\begin{eqnarray*}
U\left(\underline{c}^{k_{1}-k_{0}};k_{1}\right)-U\left(c;k_{0}\right) & \geq & U\left(\underline{c}^{k_{1}-k_{0}};k_{1}\right)-U\left(c^{k_{1}-k_{0}};k_{0}\right)\\
 & = & \int_{0}^{k_{1}-k_{0}}e^{-\rho t}\left[u\left(c^{k_{1}-k_{0}}\left(t\right)+1\right)-u\left(c^{k_{1}-k_{0}}\left(t\right)\right)\right]\mbox{d}t\\
 & \geq & \int_{0}^{k_{1}-k_{0}}e^{-\rho t}u'\left(c^{k_{1}-k_{0}}\left(t\right)+1\right)\mbox{d}t\\
 & \geq & u'\left(N\left(k_{0},k_{1}-k_{0}\right)+1\right)\int_{0}^{k_{1}-k_{0}}e^{-\rho t}\mbox{d}t
\end{eqnarray*}
which concludes the proof.\end{proof}
\begin{rem}
In the previous Lemma, the property $\eqref{orbite in sec Lemma}$
can also be proved with the ``comparison technique'', like we did
for the admissibility of $c^{T}$ in Lemma $\ref{controlli limitati e meglio}$. 

More generally, it can be proved that
\[
k\left(\cdot;k_{1},c_{H}\right)>k\left(\cdot;k_{0},c\right)
\]
where $k_{1}>k_{0}\geq0$, $c\in\mathcal{L}_{loc}^{1}\left(\left[0,+\infty\right),\mathbb{R}\right)$
and
\[
c_{H}\left(t\right):=\begin{cases}
c\left(t\right)+H & \mbox{ if }t\in\left[0,\delta_{H}\right)\\
c\left(t\right) & \mbox{ if }t\geq\delta_{H}
\end{cases}
\]
and $\delta_{H}>0$ satisfying $\delta_{H}\cdot H\leq k_{1}-k_{0}$. 

Indeed, set $k_{H}:=k\left(\cdot;k_{1},c_{H}\right)$ and $k:=k\left(\cdot;k_{0},c\right)$
and suppose by contradiction that $-\infty<\inf\left\{ t>0/k_{H}\left(t\right)\leq k\left(t\right)\right\} =:\tau$.
Then for a suitable, positive continuous function $h:\left[0,+\infty\right)\to\mathbb{R}$,
the following equality holds:
\begin{eqnarray*}
k_{H}\left(\tau\right)-k\left(\tau\right) & = & e^{\int_{0}^{\tau}h}\left[k_{1}-k_{0}+\int_{0}^{\tau}\left(c\left(s\right)-c_{H}\left(s\right)\right)e^{-\int_{0}^{s}h}\mbox{d}s\right].
\end{eqnarray*}
Moreover $\tau\leq\delta_{H}$, because on the contrary by definition
of infimum we would have $k_{H}>k$ in $\left[0,\delta_{H}\right]$;
then remembering $\eqref{eq: comp ODE strong}$ and the definition
of $c_{H}$ we would conclude that $k_{H}>k$ everywhere in $\left[0,+\infty\right)$,
which contradicts $\tau>-\infty$. Then the above equality implies
\[
k_{H}\left(\tau\right)-k\left(\tau\right)\,>\, k_{1}-k_{0}-H\int_{0}^{\tau}e^{-\int_{0}^{s}h}\mbox{d}s\,>\, k_{1}-k_{0}-\tau H\,\geq\, k_{1}-k_{0}-\delta_{H}H\,\geq0.
\]
At the same time $k_{H}\left(\tau\right)\leq k\left(\tau\right)$
by the continuity of $k_{h}$ and $k$ and by definition of infimum
(in fact the equality holds, again by continuity); hence we have reached
the desired contradiction.
\end{rem}
Now we state a simple characterisation of the admissible constant
controls.
\begin{prop}
\label{caratt controlli ammiss costanti}Let $k_{0},c\geq0.$ Then

i) $k\left(\cdot;k_{0},F\left(k_{0}\right)\right)\equiv k_{0}$

ii) the function constantly equal to $c$ is admissible at $k_{0}$
(which we write $c\in\Lambda\left(k_{0}\right)$) if, and only if
\[
c\in\left[0,F\left(k_{0}\right)\right].
\]
In particular the null function is admissible at any initial state
$k_{0}\geq0$.\end{prop}
\begin{proof}
i) By the uniqueness of the orbit.

ii)$\left(\Longleftarrow\right)$ In the first place, observe that
$F\left(k_{0}\right)\in\Lambda\left(k_{0}\right)$, by i). In the
second place, assume $c\in\left[0,F\left(k_{0}\right)\right)$ and
set $k:=k\left(\cdot;k_{0},c\right)$. Hence
\[
\dot{k}\left(0\right)=F\left(k_{0}\right)-c>0
\]
which means, by the continuity of $\dot{k}$, that we can find $\delta>0$
such that $k$ is strictly increasing in $\left[0,\delta\right]$.
In particular $\dot{k}\left(\delta\right)=F\left(k\left(\delta\right)\right)-c>F\left(k_{0}\right)-c$
because $F$ is strictly increasing too. By the fact that $\dot{k}\left(\delta\right)>0$
we see that there exists $\hat{\delta}>\delta$ such that $k$ is
strictly increasing in $\left[0,\hat{\delta}\right]$ - and so on.
Hence $k$ is strictly increasing in $\left[0,+\infty\right)$ and
in particular $k\geq0$. This shows that $c\in\Lambda\left(k_{0}\right)$.

$\left(\Longrightarrow\right)$ Suppose that $c>F\left(k_{0}\right)$
and set again $k:=k\left(\cdot;k_{0},c\right)$. Then 
\[
\dot{k}\left(0\right)=F\left(k_{0}\right)-c<0
\]
 so that we can find $\delta>0$ such that $k$ is strictly decreasing
in $\left[0,\delta\right]$, and $\dot{k}\left(\delta\right)=F\left(k\left(\delta\right)\right)-c<F\left(k_{0}\right)-c<0$.
Hence one can arbitrarily extend the neighbourhood of $0$ in which
$\dot{k}$ is strictly less than the strictly negative constant $F\left(k_{0}\right)-c$,
which implies that
\[
\lim_{t\to+\infty}k\left(t\right)=-\infty.
\]
Hence $k$ cannot be everywhere-positive and $c\notin\Lambda\left(k_{0}\right)$. \end{proof}
\begin{cor}
\label{ammissibili crescenti}The set sequence \textup{$\left(\Lambda\left(k\right)\right)_{k\geq0}$}
is strictly increasing, that is:
\[
\Lambda\left(k_{0}\right)\subsetneq\Lambda\left(k_{1}\right)
\]

for every $0\leq k_{0}<k_{1}$.\end{cor}
\begin{proof}
For every $c\in\Lambda\left(k_{0}\right)$, $k\left(\cdot;k_{0},c\right)\leq k\left(\cdot;k_{1},c\right)$
by $\ref{eq: comp ODE weak}$, which implies the second orbit being
positive, and so $c\in\Lambda\left(k_{1}\right)$. 

On the other hand, by Proposition $\ref{caratt controlli ammiss costanti}$
and by the strict monotonicity of $F$, the constant control $\hat{c}\equiv F\left(\hat{k}\right)$
belongs to $\Lambda\left(k_{1}\right)\setminus\Lambda\left(k_{0}\right)$
for any $\hat{k}\in\left(k_{0},k_{1}\right]$. 
\end{proof}

\section{Basic qualitative properties of the value function}

Now we deal with the first problem one has to solve in order to develop
the theory: the finiteness of the value function. We start setting
a result which is analogous to the one we cleared up in Remark $\ref{F Lip}$,
and which also follows from a certain sublinearity property of the
production function $F$.
\begin{rem}
\label{F sublin 2}Set $M_{0},\hat{M}\geq0$ such that:
\begin{alignat*}{1}
 & \forall x\geq M_{0}:F\left(x\right)\leq\left(L+\epsilon_{0}\right)x\\
 & \hat{M}:=\max_{\left[0,M_{0}\right]}F.
\end{alignat*}
(which is possible because $\lim_{x\to+\infty}\frac{F\left(x\right)}{x}=L$).
Hence, for every $x\geq0$:
\[
F\left(x\right)\leq\left(L+\epsilon_{0}\right)x+\hat{M}
\]

\end{rem}
$\,$
\begin{rem}
\label{remark u concava}Since $u$ is a concave function satisfying
$u\left(0\right)=0$, $u$ is sub-additive in $\left[0,+\infty\right)$
and satisfies:
\[
\forall x>0:\forall K>1:u\left(Kx\right)\leq Ku\left(x\right)
\]
\end{rem}
\begin{lem}
\label{funzionale finito}Let $k_{0}\geq0$, and $c\in\Lambda\left(k_{0}\right)$.
Hence, setting $M\left(k_{0}\right):=1+\max\left\{ \left(L+\epsilon_{0}\right)k_{0},\hat{M}\right\} $:
\begin{align*}
i)\quad & \forall t\geq0:\int_{0}^{t}c\left(s\right)\mbox{d}s\leq tM\left(k_{0}\right)\left[1+e^{\left(L+\epsilon_{0}\right)t}\right]+\frac{M\left(k_{0}\right)}{L+\epsilon_{0}}\\
ii)\quad & \lim_{t\to+\infty}e^{-\rho t}\int_{0}^{t}u\left(c\left(s\right)\right)\mbox{d}s=0\\
iii)\quad & U\left(c;k_{0}\right)=\rho\int_{0}^{+\infty}e^{-\rho t}\int_{0}^{t}u\left(c\left(s\right)\right)\mbox{d}s\mbox{d}t\leq\gamma\left(k_{0}\right)
\end{align*}
where $\gamma\left(k_{0}\right)$ is a finite quantity depending on
$k_{0}$ and on the problem's data.\end{lem}
\begin{proof}
i) Set $\kappa:=k\left(\cdot;k_{0},c\right)$ and $M\left(k_{0}\right)$
as in the hypotheses. Observe that, by Remark $\ref{F sublin 2}$,
for every $x\geq0$:
\[
F\left(x\right)\leq\left(L+\epsilon_{0}\right)x+M\left(k_{0}\right).
\]
Fix $t\geq0$; by the state equation, we have for any $s\in\left[0,t\right]$
\[
\kappa\left(s\right)\leq k_{0}+sM\left(k_{0}\right)+\left(L+\epsilon_{0}\right)\int_{0}^{s}\kappa\left(\tau\right)\mbox{d}\tau
\]
which implies by Gronwall's inequality:
\[
\kappa\left(s\right)\leq\left[k_{0}+sM\left(k_{0}\right)\right]e^{\left(L+\epsilon_{0}\right)s}\quad\forall s\in\left[0,t\right],
\]
as $s\to k_{0}+sM\left(k_{0}\right)$ is increasing. So
\begin{eqnarray*}
\int_{0}^{t}\left(L+\epsilon_{0}\right)\kappa\left(s\right)\mbox{d}s & \leq & k_{0}\left(L+\epsilon_{0}\right)\int_{0}^{t}e^{\left(L+\epsilon_{0}\right)s}\mbox{d}s+M\left(k_{0}\right)\left(L+\epsilon_{0}\right)\int_{0}^{t}s\cdot e^{\left(L+\epsilon_{0}\right)s}\mbox{d}s\\
 & = & k_{0}e^{\left(L+\epsilon_{0}\right)t}-k_{0}+tM\left(k_{0}\right)e^{\left(L+\epsilon_{0}\right)t}-\frac{M\left(k_{0}\right)}{\left(L+\epsilon_{0}\right)}e^{\left(L+\epsilon_{0}\right)t}+\frac{M\left(k_{0}\right)}{\left(L+\epsilon_{0}\right)}\\
 & = & tM\left(k_{0}\right)e^{\left(L+\epsilon_{0}\right)t}+\left[k_{0}-\frac{M\left(k_{0}\right)}{\left(L+\epsilon_{0}\right)}\right]e^{\left(L+\epsilon_{0}\right)t}+\frac{M\left(k_{0}\right)}{\left(L+\epsilon_{0}\right)}-k_{0}\\
 & \leq & tM\left(k_{0}\right)e^{\left(L+\epsilon_{0}\right)t}+\frac{M\left(k_{0}\right)}{\left(L+\epsilon_{0}\right)}-k_{0}
\end{eqnarray*}
Hence, again by the state equation, for every $t\geq0$:
\begin{eqnarray*}
\int_{0}^{t}c\left(s\right)\mbox{d}s & = & k_{0}-\kappa\left(t\right)+\int_{0}^{t}F\left(\kappa\left(s\right)\right)\mbox{d}s\\
 & \leq & k_{0}+tM\left(k_{0}\right)+\int_{0}^{t}\left(L+\epsilon_{0}\right)\kappa\left(s\right)\mbox{d}s\,\leq\, tM\left(k_{0}\right)\left[1+e^{\left(L+\epsilon_{0}\right)t}\right]+\frac{M\left(k_{0}\right)}{\left(L+\epsilon_{0}\right)}.
\end{eqnarray*}
which proves the first assertion.

ii) In the second place, it follows by Jensen inequality, the monotonicity
of $u$ and Remark $\ref{remark u concava}$, that for every $t\geq0$:
\begin{eqnarray*}
0\leq e^{-\rho t}\int_{0}^{t}u\left(c\left(s\right)\right)\mbox{d}s\hspace{-2mm} & \leq & \hspace{-2mm}te^{-\rho t}u\left(\frac{\int_{0}^{t}c\left(s\right)\mbox{d}s}{t}\right)\,\leq\, te^{-\rho t}u\left(M\left(k_{0}\right)\left[1+e^{\left(L+\epsilon_{0}\right)t}\right]+\frac{M\left(k_{0}\right)}{t\left(L+\epsilon_{0}\right)}\right)\\
 & \leq & \hspace{-2mm}te^{-\rho t}\Biggl\{ u\left(M\left(k_{0}\right)\right)+M\left(k_{0}\right)u\left(e^{\left(L+\epsilon_{0}\right)t}\right)+u\left(\frac{M\left(k_{0}\right)}{t\left(L+\epsilon_{0}\right)}\right)\Biggl\};
\end{eqnarray*}
observe that this quantity tends to $0$ as $t\to+\infty$, particulary
by the last condition assumed in $\eqref{assumption u}$ about $u$;
so also the second claim is proven.

Finally, integrating by parts, and using ii)
\begin{eqnarray*}
U\left(c;k_{0}\right) & = & \int_{0}^{+\infty}e^{-\rho t}u\left(c\left(t\right)\right)\mbox{d}t\\
 & = & \lim_{T\to+\infty}\Biggl\{ e^{-\rho T}\int_{0}^{T}u\left(c\left(s\right)\right)\mbox{d}s+\rho\int_{0}^{T}e^{-\rho t}\int_{0}^{t}u\left(c\left(s\right)\right)\mbox{d}s\mbox{d}t\Biggl\}\\
 & = & \rho\int_{0}^{+\infty}e^{-\rho t}\int_{0}^{t}u\left(c\left(s\right)\right)\mbox{d}s\mbox{d}t\\
 & \leq & \rho\int_{0}^{+\infty}te^{-\rho t}\Biggl\{ u\left(M\left(k_{0}\right)\right)+M\left(k_{0}\right)u\left(e^{\left(L+\epsilon_{0}\right)t}\right)+u\left(\frac{M\left(k_{0}\right)}{t\left(L+\epsilon_{0}\right)}\right)\Biggl\}\mbox{d}t\\
 & \leq & \rho u\left(M\left(k_{0}\right)\right)\int_{0}^{+\infty}te^{-\rho t}\mbox{d}t+\rho M\left(k_{0}\right)\int_{0}^{+\infty}te^{-\rho t}u\left(e^{\left(L+\epsilon_{0}\right)t}\right)\mbox{d}t\\
 &  & +\rho u\left(\frac{M\left(k_{0}\right)}{L+\epsilon_{0}}\right)\left\{ \int_{0}^{1}e^{-\rho t}\mbox{d}t+\int_{1}^{+\infty}te^{-\rho t}\mbox{d}t\right\} 
\end{eqnarray*}
Now it is sufficient to observe that by Remark $\ref{remark 1}$ this
upper bound is finite and set it equal to $\gamma\left(k_{0}\right)$.
\end{proof}
Hence we have established the starting point of the theory.
\begin{cor}
The value function $V:\left[0,+\infty\right)\to\mathbb{R}$ is well-definite;
that is, for every $k_{0}\geq0$, $V\left(k_{0}\right)<+\infty$.\end{cor}
\begin{proof}
Take $k_{0}\geq0$ and set $\gamma\left(k_{0}\right)$ as in Lemma
$\ref{funzionale finito}$. Hence:
\[
V\left(k_{0}\right)=\sup_{c\in\Lambda\left(k_{0}\right)}U\left(c;k_{0}\right)\leq\gamma\left(k_{0}\right)<+\infty.
\]

\end{proof}
Next, we prove directly some useful asymptotic properties of the value
function.
\begin{thm}[\textbf{Asymptotic properties of the value function} ]
\label{Basic Prop V}The value function $V:\left[0,+\infty\right)\to\mathbb{R}$
satisfies:
\begin{eqnarray*}
i) &  & \lim_{k\to+\infty}V\left(k\right)=+\infty\\
ii) &  & \lim_{k\to+\infty}\frac{V\left(k\right)}{k}=0\\
iii) &  & \lim_{k\to0}V\left(k\right)=V\left(0\right)=0
\end{eqnarray*}
\end{thm}
\begin{proof}
\textbf{i)} For every $k_{0}\geq0$ the constant control $F\left(k_{0}\right)$
is admissible at $k_{0}$ by Proposition $\ref{caratt controlli ammiss costanti}$;
hence
\[
V\left(k_{0}\right)\geq U\left(F\left(k_{0}\right);k_{0}\right)=\frac{u\left(F\left(k_{0}\right)\right)}{\rho}\to+\infty
\]
as $k_{0}\to+\infty$, by the assumptions on $u$ and $F$.\\

\textbf{ii)} Set $\hat{M}>0$ as in Remark $\ref{F sublin 2}$ and
$k_{0}>0$ such that:
\begin{equation}
k_{0}>\frac{1}{L+\epsilon_{0}}\hat{M}\label{condiz per V infty}
\end{equation}
Hence, for every $x>0$:
\begin{equation}
F\left(x\right)\leq\left(L+\epsilon_{0}\right)\left(x+k_{0}\right)\label{sublin per V infty}
\end{equation}

By reasons that will be clear later, suppose also that:
\begin{equation}
k_{0}>\frac{1}{L+\epsilon_{0}}\label{eq: ulteriore su k_0}
\end{equation}
Observe that the proof of Lemma $\ref{funzionale finito}$, i) does
not require $M\left(k_{0}\right)\geq1$, but only $M\left(k_{0}\right)\geq\hat{M}$;
hence $\eqref{condiz per V infty}$ and $\eqref{sublin per V infty}$
imply that the property in Lemma $\ref{funzionale finito}$, i) holds
for $M\left(k_{0}\right)=k_{0}\left(L+\epsilon_{0}\right)$ - which
means that:
\begin{equation}
\forall t\geq0:\int_{0}^{t}c\left(s\right)\mbox{d}s\leq k_{0}+tk_{0}\left(L+\epsilon_{0}\right)\left[1+e^{\left(L+\epsilon_{0}\right)t}\right].\label{Ic per V infty}
\end{equation}
In particular
\begin{equation}
\forall t\geq1:\frac{\int_{0}^{t}c\left(s\right)\mbox{d}s}{t}\leq k_{0}+k_{0}\left(L+\epsilon_{0}\right)+k_{0}\left(L+\epsilon_{0}\right)e^{\left(L+\epsilon_{0}\right)t}.\label{Ic/t per V infty}
\end{equation}
Now set 
\begin{equation}
J_{c}\left(\alpha,\beta\right):=\int_{\alpha}^{\beta}te^{-\rho t}u\left(\frac{\int_{0}^{t}c\left(s\right)\mbox{d}s}{t}\right)\mbox{d}t\label{def Jc}
\end{equation}

and fix $N>0$ .

We provide three different estimates, over $J_{c}\left(0,1\right)$,
$J_{c}\left(1,N\right)$ and $J_{c}\left(N,+\infty\right)$, using
Remark $\ref{remark u concava}$.

First, we have by $\eqref{Ic per V infty}$:
\begin{eqnarray*}
J_{c}\left(0,1\right) & \leq & \int_{0}^{1}te^{-\rho t}\frac{1}{t}u\left(\int_{0}^{1}c\left(s\right)\mbox{d}s\right)\mbox{d}t\\
 & \leq & u\left[k_{0}\left(1+\left(L+\epsilon_{0}\right)\left(1+e^{\left(L+\epsilon_{0}\right)}\right)\right)\right]\frac{1-e^{-\rho}}{\rho}\\
 & \leq & u\left(k_{0}\right)\frac{1-e^{-\rho}}{\rho}\left[1+\left(L+\epsilon_{0}\right)\left(1+e^{\left(L+\epsilon_{0}\right)}\right)\right].
\end{eqnarray*}
Moreover, by $\eqref{Ic/t per V infty}$:
\begin{eqnarray*}
J_{c}\left(1,N\right) & \leq & \int_{1}^{N}te^{-\rho t}u\left(k_{0}+k_{0}\left(L+\epsilon_{0}\right)+k_{0}\left(L+\epsilon_{0}\right)e^{\left(L+\epsilon_{0}\right)t}\right)\mbox{d}t\\
 & \leq & u\left(k_{0}+k_{0}\left(L+\epsilon_{0}\right)\right)\int_{1}^{N}te^{-\rho t}\mbox{d}t+u\left(k_{0}\left(L+\epsilon_{0}\right)\right)\int_{1}^{N}te^{-\rho t}e^{\left(L+\epsilon_{0}\right)t}\mbox{d}t\\
 & \leq & u\left[k_{0}\left(1+L+\epsilon_{0}\right)\right]\left(1+e^{\left(L+\epsilon_{0}\right)N}\right)\int_{1}^{N}te^{-\rho t}\mbox{d}t
\end{eqnarray*}
Finally, remembering that $k_{0}\left(L+\epsilon_{0}\right)>1$ by
$\eqref{eq: ulteriore su k_0}$,
\begin{eqnarray*}
J_{c}\left(N,+\infty\right) & \leq & \int_{N}^{+\infty}te^{-\rho t}u\left(k_{0}+k_{0}\left(L+\epsilon_{0}\right)+k_{0}\left(L+\epsilon_{0}\right)e^{\left(L+\epsilon_{0}\right)t}\right)\mbox{d}t\\
 & \leq & u\left(k_{0}+k_{0}\left(L+\epsilon_{0}\right)\right)\int_{N}^{+\infty}te^{-\rho t}\mbox{d}t+k_{0}\left(L+\epsilon_{0}\right)\int_{N}^{+\infty}te^{-\rho t}u\left(e^{\left(L+\epsilon_{0}\right)t}\right)\mbox{d}t
\end{eqnarray*}

Now we show that 
\[
\lim_{k\to+\infty}\frac{V\left(k\right)}{k}=0.
\]
Fix $\eta>0$; by Remark $\ref{remark 1}$, we can chose $N_{\eta}>0$
such that
\[
\left(L+\epsilon_{0}\right)\int_{N_{\eta}}^{+\infty}te^{-\rho t}u\left(e^{\left(L+\epsilon_{0}\right)t}\right)\mbox{d}t<\eta.
\]

Hence for $k_{0}$ satisfying:
\[
k_{0}>\max\left\{ \frac{1}{L+\epsilon_{0}}\hat{M},\frac{1}{L+\epsilon_{0}}\right\} 
\]
and for \emph{every} $c\in\Lambda\left(k_{0}\right)$, the above estimates
imply: 
\begin{eqnarray}
U\left(c;k_{0}\right) & = & \rho\int_{0}^{+\infty}e^{-\rho t}\int_{0}^{t}u\left(c\left(s\right)\right)\mbox{d}s\mbox{d}t\nonumber \\
 & \leq & \rho J_{c}\left(0,1\right)+\rho J_{c}\left(1,N_{\eta}\right)+\rho J_{c}\left(N_{\eta},+\infty\right)\nonumber \\
 & \leq & u\left(k_{0}\right)\left(1-e^{-\rho}\right)\left[1+\left(L+\epsilon_{0}\right)\left(e^{\left(L+\epsilon_{0}\right)}+1\right)\right]+\nonumber \\
 &  & +u\left(k_{0}\right)\left(1+L+\epsilon_{0}\right)\left(1+e^{\left(L+\epsilon_{0}\right)N_{\eta}}\right)\int_{1}^{N_{\eta}}te^{-\rho t}\mbox{d}t+\nonumber \\
 &  & +u\left(k_{0}\right)\left(1+L+\epsilon_{0}\right)\int_{N_{\eta}}^{+\infty}te^{-\rho t}\mbox{d}t+k_{0}\eta\label{eq: nuova prop V}
\end{eqnarray}
following Remark $\ref{remark u concava}$, Lemma $\ref{funzionale finito}$,
iii), $\eqref{def Jc}$ and Jensen inequality. Now observe that:
\[
\lim_{k_{0}\to+\infty}\frac{u\left(k_{0}\right)}{k_{0}}=\lim_{k_{0}\to+\infty}u'\left(k_{0}\right)=0.
\]
Hence for $k_{0}$ sufficiently large (say $k_{0}>k^{*}$):
\begin{eqnarray*}
\frac{u\left(k_{0}\right)}{k_{0}} & < & \eta\Biggl\{\left(1-e^{-\rho}\right)\left[1+\left(L+\epsilon_{0}\right)\left(e^{\left(L+\epsilon_{0}\right)}+1\right)\right]+\\
 &  & +\left(1+L+\epsilon_{0}\right)\left(1+e^{\left(L+\epsilon_{0}\right)N_{\eta}}\right)\int_{1}^{N_{\eta}}te^{-\rho t}\mbox{d}t+\left(1+L+\epsilon_{0}\right)\int_{N_{\eta}}^{+\infty}te^{-\rho t}\mbox{d}t\Biggl\}^{-1}
\end{eqnarray*}
Observe that this is possible because the expression into the brackets
does not depend on $k_{0}$. In fact, like $N_{\eta},$ it depends
only on $\eta$ and on the problem's data $L$, $\epsilon_{0}$, $\rho$
- and so does $k^{*}$.

By $\eqref{eq: nuova prop V}$, this implies for \emph{every} $c\in\Lambda\left(k_{0}\right)$:
\[
U\left(c;k_{0}\right)\leq2k_{0}\eta
\]
which gives, taking the sup over $\Lambda\left(k_{0}\right)$:
\[
V\left(k_{0}\right)\leq2k_{0}\eta.
\]
Hence the assertion is proven, because the previous inequality holds
for every 
\[
k_{0}>\max\left\{ \frac{1}{L+\epsilon_{0}}\hat{M},\frac{1}{L+\epsilon_{0}},k^{*}\right\} ,
\]
and the last quantity is a threshold depending only on $\eta$ and
on the problem's data.\\

\textbf{iii)} In the first place, we prove that
\[
V\left(0\right)=0.
\]
 Let $c\in\Lambda\left(0\right);$ by definition, $c\geq0$ so that
\[
\forall t\geq0:\dot{k}\left(t;0,c\right)\leq F\left(k\left(t;0,c\right)\right).
\]
Observe that $F$ is precisely the function which defines the dynamics
of $k\left(\cdot;0,0\right)$, hence by $\eqref{eq: comp ODE weak}$:
\[
\forall t\geq0:k\left(t;0,c\right)\leq k\left(t;0,0\right)=0
\]

where the last equality holds by Lemma $\ref{caratt controlli ammiss costanti}$,
i).

Hence $k\left(\cdot;0,c\right)\equiv0$ which together with $F\left(0\right)=0$
implies $c\equiv0$. So $\Lambda\left(0\right)=\left\{ 0\right\} $,
which implies
\[
V\left(0\right)=U\left(0;0\right)=\int_{0}^{+\infty}e^{-\rho t}u\left(0\right)\mbox{d}t=0
\]

Now we show that
\[
\lim_{k\to0}V\left(k\right)=0.
\]
In this case we have to study the behaviour of $V\left(k_{0}\right)$
when $k_{0}\to0$, so we use the sublinearity of $F\left(x\right)$
for $x\to+\infty$ and the concavity of $F$ near $0$.

As a first step, we construct a linear function which is always above
$F$ with these two tools. Indeed we show that there is $m>0$ such
that the function
\[
G\left(x\right):=\begin{cases}
mx & \mbox{ if }x\in\left[0,\bar{k}\right]\\
\left(L+\epsilon_{0}\right)\left(x-\bar{k}\right)+m\bar{k} & \mbox{ if }x\geq\bar{k}
\end{cases}
\]
satisfies
\begin{equation}
\forall x\geq0:F\left(x\right)\leq G\left(x\right).\label{F<G}
\end{equation}
If $F'\left(\bar{k}\right)\leq L+\epsilon_{0}$ then it is enough
to take $m>\max\left\{ F'\left(0\right),F'\left(\bar{k}\right),\frac{F\left(\bar{k}\right)}{\bar{k}}\right\} $.

If $F'\left(\bar{k}\right)>L+\epsilon_{0}$ then observe that for
every $x\geq\bar{k}$:
\[
\frac{F\left(x\right)-F\left(\bar{k}\right)}{x-\bar{k}}\leq F'\left(x\right)\to L,\ \mbox{for }x\to+\infty.
\]
Hence there exists $M\geq\bar{k}$ such that, for every $x\geq M$,
\[
F\left(x\right)\leq F\left(\bar{k}\right)+\left(L+\epsilon_{0}\right)\left(x-\bar{k}\right)
\]
which implies that for every $x\geq\bar{k}$:
\[
F\left(x\right)\leq\left(L+\epsilon_{0}\right)\left(x-\bar{k}\right)+F\left(\bar{k}\right)+\max_{\left[\bar{k},M\right]}F.
\]
Hence we replace the third of the previous conditions on $m$ with
$m\bar{k}>F\left(\bar{k}\right)+\max_{\left[\bar{k},M\right]}F$.
Observe that condition $m>F'\left(\bar{k}\right)$ is still necessary
to ensure that $mx>F\left(x\right)$ for $x\in\left[\underline{k},\bar{k}\right]$
(Lagrange's theorem proves that it is sufficient).

Suppose also, for reasons that will be clear later, that
\begin{equation}
m>1.\label{m>1}
\end{equation}

Now take $k_{0}>0$, $c\in\Lambda\left(k_{0}\right)$ and consider
the function $h:[0,+\infty)\to\mathbb{R}$ which is the unique solution
to the Cauchy's problem
\[
\begin{cases}
h\left(0\right)=k_{0}\\
\dot{h}\left(t\right)=G\left(h\left(t\right)\right) & \ t\geq0
\end{cases}
\]
Hence, by $\eqref{F<G}$ and $\eqref{eq: comp ODE weak}$, $k:=k\left(\cdot;k_{0},c\right)\leq h.$
So, setting 
\begin{align*}
 & \bar{t}:=\frac{1}{m}\log\left(\frac{\bar{k}}{k_{0}}\right)\mbox{ and }\hat{k}:=\bar{k}\left(m-L-\epsilon_{0}\right)
\end{align*}
we get, for every $t\in\left[0,\bar{t}\right]$:
\[
h\left(t\right)=k_{0}e^{mt}
\]
and, for every $t\geq\bar{t}$:
\begin{eqnarray*}
h\left(t\right) & = & e^{\left(L+\epsilon_{0}\right)t}\int_{\bar{t}}^{t}e^{-\left(L+\epsilon_{0}\right)s}\hat{k}\mbox{d}s+\bar{k}e^{-\left(L+\epsilon_{0}\right)\bar{t}}\ =\ \frac{\hat{k}e^{-\left(L+\epsilon_{0}\right)\bar{t}}}{L+\epsilon_{0}}e^{\left(L+\epsilon_{0}\right)t}+\bar{k}e^{-\left(L+\epsilon_{0}\right)\bar{t}}-\frac{\hat{k}}{L+\epsilon_{0}}\\
 & =: & \omega_{0}\left(k_{0}\right)e^{\left(L+\epsilon_{0}\right)t}+\omega_{1}\left(k_{0}\right)-\frac{\hat{k}}{L+\epsilon_{0}}
\end{eqnarray*}
where by definition of $\bar{t}$ the functions $\omega_{i}$ satisfy:
\begin{align*}
 & \omega_{0}\left(k_{0}\right)=\frac{\hat{k}}{L+\epsilon_{0}}e^{-\left(L+\epsilon_{0}\right)\bar{t}}=\frac{\hat{k}}{L+\epsilon_{0}}\left(\frac{k_{0}}{\overline{k}}\right)^{\frac{L+\epsilon_{0}}{m}}\\
 & \omega_{1}\left(k_{0}\right)=\overline{k}e^{-\left(L+\epsilon_{0}\right)\bar{t}}=\overline{k}\left(\frac{k_{0}}{\overline{k}}\right)^{\frac{L+\epsilon_{0}}{m}}.
\end{align*}
Using the state equation, we deduce two estimates for the integrals
of $c$ by the above computations of $h$. 

For every $t\in\left[0,\bar{t}\right]$ (remembering that $h$ is
increasing so that $\forall s\leq t:\, h\left(s\right)\leq\bar{k}$):
\begin{eqnarray}
\int_{0}^{t}c\left(s\right)\mbox{d}s & \leq & k_{0}+\int_{0}^{t}F\left(k\left(s\right)\right)\mbox{d}s\leq k_{0}+\int_{0}^{t}G\left(h\left(s\right)\right)\mbox{d}s\nonumber \\
 & = & k_{0}+\int_{0}^{t}mk_{0}e^{ms}\mbox{d}s=k_{0}e^{mt}.\label{stima 1 per V in 0}
\end{eqnarray}
Instead, for every $t>\bar{t}$:
\begin{eqnarray}
\int_{0}^{t}c\left(s\right)\mbox{d}s & \leq & k_{0}+\int_{0}^{\bar{t}}G\left(h\left(s\right)\right)\mbox{d}s+\int_{\bar{t}}^{t}G\left(h\left(s\right)\right)\mbox{d}s\nonumber \\
 & \leq & k_{0}e^{m\bar{t}}+\int_{\bar{t}}^{t}\left\{ \left(L+\epsilon_{0}\right)h\left(s\right)+\hat{k}\right\} \mbox{d}s\nonumber \\
 & \leq & \bar{k}+\left(t-\bar{t}\right)\hat{k}+\left(L+\epsilon_{0}\right)\int_{\bar{t}}^{t}\left\{ \omega_{0}\left(k_{0}\right)e^{\left(L+\epsilon_{0}\right)s}+\omega_{1}\left(k_{0}\right)-\frac{\hat{k}}{L+\epsilon_{0}}\right\} \mbox{d}s\nonumber \\
 & \leq & \bar{k}+\omega_{0}\left(k_{0}\right)\left[e^{\left(L+\epsilon_{0}\right)t}-e^{\left(L+\epsilon_{0}\right)\bar{t}}\right]+\left(L+\epsilon_{0}\right)\left(t-\bar{t}\right)\omega_{1}\left(k_{0}\right)\nonumber \\
 & \leq & \bar{k}+\omega_{0}\left(k_{0}\right)e^{\left(L+\epsilon_{0}\right)t}-\frac{\hat{k}}{L+\epsilon_{0}}+\left(L+\epsilon_{0}\right)\left(t-\bar{t}\right)\omega_{1}\left(k_{0}\right)\label{stima 2 per V in 0}
\end{eqnarray}
where we have used $h\left(s\right)\geq\bar{k}$ for $s\in\left(\bar{t},t\right)$
and the fact that $k_{0}e^{m\bar{t}}=\bar{k}$.

Now observe that
\begin{align}
 & \lim_{k_{0}\to0}\omega_{0}\left(k_{0}\right)=\lim_{k_{0}\to0}\omega_{1}\left(k_{0}\right)=0\nonumber \\
 & \lim_{k_{0}\to0}\bar{t}=\lim_{k_{0}\to0}\frac{1}{m}\log\left(\frac{\bar{k}}{k_{0}}\right)=+\infty.\label{condizioni V in 0}
\end{align}
Hence if $k_{0}$ is small enough (say $k_{0}<k^{*}$), we may assume
$\bar{t}>1$ and $\omega_{i}\left(k_{0}\right)\leq1$ for $i=0,1$,
so that $\eqref{stima 2 per V in 0}$ implies, for every $t>\bar{t}$:
\begin{eqnarray}
\frac{\int_{0}^{t}c\left(s\right)\mbox{d}s}{t} & \leq & \bar{k}+e^{\left(L+\epsilon_{0}\right)t}+\left(L+\epsilon_{0}\right)\frac{\left(t-\bar{t}\right)}{t}\,\leq\bar{k}+e^{\left(L+\epsilon_{0}\right)t}+\left(L+\epsilon_{0}\right)\label{stima 3 per V in 0}
\end{eqnarray}
Hence, by Lemma $\ref{funzionale finito}$, iii) , by Remark $\ref{remark u concava}$,
and by $\eqref{stima 1 per V in 0}$, $\eqref{stima 3 per V in 0}$,
the following inequality holds for every $k_{0}<k^{*}$ and every
$c\in\Lambda\left(k_{0}\right)$:
\begin{eqnarray*}
0\ \leq\ U\left(c;k_{0}\right) & = & \rho\int_{0}^{+\infty}e^{-\rho t}\int_{0}^{t}u\left(c\left(s\right)\right)\mbox{d}s\mbox{d}t\,\leq\,\rho\int_{0}^{+\infty}te^{-\rho t}u\left(\frac{\int_{0}^{t}c\left(s\right)\mbox{d}s}{t}\right)\mbox{d}t\\
 & \leq & \rho\int_{0}^{1}e^{-\rho t}u\left(\int_{0}^{t}c\left(s\right)\mbox{d}s\right)\mbox{d}t+\rho\int_{1}^{\bar{t}}te^{-\rho t}u\left(\frac{k_{0}e^{mt}}{t}\right)\mbox{d}t+\\
 &  & +\rho\int_{\bar{t}}^{+\infty}te^{-\rho t}u\left(\bar{k}+e^{\left(L+\epsilon_{0}\right)t}+\left(L+\epsilon_{0}\right)\right)\mbox{d}t\\
 & \leq & \rho\int_{0}^{1}e^{-\rho t}u\left(k_{0}e^{mt}\right)\mbox{d}t+\rho u\left(\frac{k_{0}e^{m\bar{t}}}{\bar{t}}\right)\int_{1}^{\bar{t}}te^{-\rho t}\mbox{d}t+\\
 &  & +\rho u\left(\bar{k}+\left(L+\epsilon_{0}\right)\right)\int_{\bar{t}}^{+\infty}te^{-\rho t}\mbox{d}t+\rho\int_{\bar{t}}^{+\infty}te^{-\rho t}u\left(e^{\left(L+\epsilon_{0}\right)t}\right)\mbox{d}t\\
 & \leq & \rho u\left(k_{0}e^{m}\right)\int_{0}^{1}e^{-\rho t}\mbox{d}t+\rho u\left(\frac{\bar{k}}{\bar{t}}\right)\frac{e^{-\rho}\left(1+\rho\right)}{\rho^{2}}+\\
 &  & +\rho u\left(\bar{k}+\left(L+\epsilon_{0}\right)\right)\int_{\bar{t}}^{+\infty}te^{-\rho t}\mbox{d}t+\rho\int_{\bar{t}}^{+\infty}te^{-\rho t}u\left(e^{\left(L+\epsilon_{0}\right)t}\right)\mbox{d}t
\end{eqnarray*}
where we used also the fact that the function $t\to\frac{e^{mt}}{t}$
is increasing for $t>1$, by condition $\eqref{m>1}$. 

It follows from $\eqref{condizioni V in 0}$ and the fact that $\lim_{x\to0}u\left(x\right)=0$,
together with Remark $\ref{remark 1}$, that the above quantity tends
to $0$ as $k_{0}\to0$; moreover, that quantity does not depend on
$c$.

Hence, noticing that $k^{*}$ depends only on the data and $m$, we
see that for any $\epsilon>0$ there exists $\delta\in\left(0,k^{*}\right]$
such that for every $k_{0}\in\left(0,\delta\right)$ and for \emph{every}
$c\in\Lambda\left(k_{0}\right)$: 
\[
U\left(c;k_{0}\right)\leq\epsilon,
\]
which implies, taking the sup over $\Lambda\left(k_{0}\right)$, that
$V\left(k_{0}\right)\leq\epsilon$ - and the assertion follows.
\end{proof}

\section{Existence of the optimal control}

In this section we deal with a fundamental topic of any optimization
problem: the existence of an optimal control. For any fixed $k_{0}\geq0$,
we look for a control $c^{*}\in\Lambda\left(k_{0}\right)$ satisfying
\[
U\left(c^{*};k_{0}\right)=\sup_{c\in\Lambda\left(k_{0}\right)}U\left(c;k_{0}\right)=V\left(k_{0}\right).
\]
We preliminary observe that the peculiar features of our problem,
particularly the absence of any boundedness conditions on the admissible
controls, force us to make use of this result in proving certain properties
of the value function which usually do not require such a settlement
- and which we postpone for this reason.

First observe that by Theorem $\ref{Basic Prop V}$, iii) if we set
$c_{0}:\equiv0$, then $U\left(c_{0},0\right)=0=V\left(0\right)$;
hence $c_{0}$ is optimal at $0$.

Let $k_{0}>0$; this will be the initial state which we will refer
to during the whole section - hence the meaning of this symbol will
not change in this context.

We split the construction in various steps; first we make a simple
but important
\begin{rem}
\label{remark weak ovunque}Suppose that $\left(f_{n}\right)_{n\in\mathbb{N}},f$
are functions in $\mathcal{L}_{loc}^{1}\left(\left[0,+\infty\right),\mathbb{R}\right)$
such that for \emph{every} $N\in\mathbb{N}$, $f_{n}\rightharpoonup f$
in $L^{1}\left(\left[0,N\right],\mathbb{R}\right)$. If $T>0$, $T\in\mathbb{R}$,
then it follows from the definition of weak convergence that, for
$g\in L^{\infty}\left(\left[0,T\right],\mathbb{R}\right)$: 
\begin{eqnarray*}
\int_{0}^{T}g\left(s\right)f_{n}\left(s\right)\mbox{d}s\hspace{-2mm} & =\hspace{-2mm} & \int_{0}^{\left[T\right]+1}\chi_{\left[0,T\right]}g\left(s\right)f_{n}\left(s\right)\mbox{d}s\to\int_{0}^{\left[T\right]+1}\chi_{\left[0,T\right]}g\left(s\right)f\left(s\right)\mbox{d}s=\int_{0}^{T}g\left(s\right)f\left(s\right)\mbox{d}s.
\end{eqnarray*}
Hence $f_{n}\rightharpoonup f$ in $L^{1}\left(\left[0,T\right],\mathbb{R}\right)$,
for every $T>0,T\in\mathbb{R}$.
\end{rem}
\textbf{\emph{$\,$}}\\
\textbf{\emph{Step 1.}} The first step is to find a maximizing sequence
of controls which are admissible at $k_{0}$ and a function $\gamma\in\mathcal{L}_{loc}^{1}\left(\left[0,+\infty\right),\mathbb{R}\right)$,
such that the sequence weakly converges to $\gamma$ in $L^{1}\left(\left[0,T\right],\mathbb{R}\right)$,
for every $T>0$.

By definition of supremum, we can find a maximizing sequence; that
is to say, there exist a sequence $\left(c_{n}\right)_{n\in\mathbb{N}}\subseteq\Lambda\left(k_{0}\right)$
of admissible controls satisfying:
\[
\lim_{n\to+\infty}U\left(c_{n};k_{0}\right)=V\left(k_{0}\right).
\]
In order to apply the tools we set up at the beginning of the chapter,
we need the following result.
\begin{lem}
\label{Dunf-Pett}Let $T\in\mathbb{N}$ and $\left(f_{n}\right)_{n\in\mathbb{N}}\subseteq\mathcal{L}_{loc}^{1}\left(\left[0,+\infty\right),\mathbb{R}\right)$,
$M\left(T\right)>0$ such that
\[
\forall n\in\mathbb{N}:\left\Vert f_{n}\right\Vert _{\infty,\left[0,T\right]}\leq M\left(T\right).
\]
Then there exist a subsequence $\left(\overline{f}_{n}\right)_{n\in\mathbb{N}}$
of$\,$ $\left(f_{n}\right)_{n\in\mathbb{N}}$ and a function $f\in L^{1}\left(\left[0,T\right],\mathbb{R}\right)$
such that\textup{
\[
\overline{f}_{n}\rightharpoonup f\mbox{ in }L^{1}\left(\left[0,T\right],\mathbb{R}\right).
\]
}\end{lem}
\begin{proof}
For every $0\leq t_{0}<t_{1}\leq T$:
\[
\int_{t_{0}}^{t_{1}}\left|f_{n}\left(s\right)\right|\mbox{d}s\leq\left\Vert f_{n}\right\Vert _{\infty,\left[0,T\right]}\cdot\left(t_{1}-t_{0}\right)\leq M\left(T\right)\cdot\left(t_{1}-t_{0}\right).
\]
Hence, by the fact that the family $\left\{ \left(t_{0},t_{1}\right)\in\mathcal{P}\left(\left[0,T\right]\right)/t_{0},t_{1}\in\left[0,T\right]\right\} $
generates the Borel $\sigma$- algebra in $\left[0,T\right]$, we
deduce that the latter relation holds for every measurable set $E\subseteq\left[0,T\right]$;
that is to say
\[
\int_{E}\left|f_{n}\left(s\right)\right|\mbox{d}s\leq M\left(T\right)\cdot\mu\left(E\right).
\]
This implies easily that the densities $\left\{ d_{n}/n\in\mathbb{N}\right\} $
given by $d_{n}\left(E\right):=\int_{E}f_{n}\left(s\right)\mbox{d}s$
are absolutely equicontinuous. So the thesis follows from the Dunford-Pettis
criterion (see \cite{Edwards}). Observe that the third condition
required by such theorem, that is to say, for any $\epsilon>0$ there
exists a compact set $K_{\epsilon}\subseteq\left[0,T\right]$ such
that
\[
\forall n\in\mathbb{N}:\int_{\left[0,T\right]\setminus K_{\epsilon}}f_{n}\left(s\right)\mbox{d}s\leq\epsilon
\]
is obviously satisfied.
\end{proof}
Now we apply Lemma $\ref{controlli limitati e meglio}$ to $\left(c_{n}\right)_{n\in\mathbb{N}}$
in order to find a new sequence $\left(c_{n}^{1}\right)_{n\in\mathbb{N}}\subseteq\Lambda\left(k_{0}\right)$
such that, for every $n\in\mathbb{N}$:
\begin{align*}
 & U\left(c_{n}^{1};k_{0}\right)\geq U\left(c_{n};k_{0}\right)\\
 & c_{n}^{1}=c_{n}\wedge N\left(k_{0},1\right)\mbox{ a.e. in }\left[0,1\right].
\end{align*}
In particular $\left(c_{n}^{1}\right)_{n\in\mathbb{N}}\subseteq\mathcal{L}_{loc}^{1}\left(\left[0,+\infty\right),\mathbb{R}\right)$
and $\left\Vert c_{n}^{1}\right\Vert _{\infty,\left[0,1\right]}\leq N\left(k_{0},1\right)$
for every $n\in\mathbb{N}$. Hence by Lemma $\ref{Dunf-Pett}$, there
exists a sequence $\left(\overline{c}_{n}^{1}\right)_{n\in\mathbb{N}}$
extracted from $\left(c_{n}^{1}\right)_{n\in\mathbb{N}}$ and a function
$c^{1}\in L^{1}\left(\left[0,1\right],\mathbb{R}\right)$ such that
\[
\overline{c}_{n}^{1}\rightharpoonup c^{1}\mbox{ in }L^{1}\left(\left[0,1\right],\mathbb{R}\right).
\]

Now define, for every $n\in\mathbb{N}$:
\[
c_{n}^{2}:=\left(\overline{c}_{n}^{1}\right)^{2}
\]
where $\left(\overline{c}_{n}^{1}\right)^{2}$ is understood with
the notation of Lemma $\ref{controlli limitati e meglio}$.

Hence for every $n\in\mathbb{N}$:
\begin{align*}
 & U\left(c_{n}^{2};k_{0}\right)\geq U\left(\overline{c}_{n}^{1};k_{0}\right)\\
 & c_{n}^{2}=\overline{c}_{n}^{1}\wedge N\left(k_{0},2\right)\mbox{ a.e. in }\left[0,2\right].
\end{align*}
Again by Lemma $\ref{Dunf-Pett}$, we can exhibit a subsequence $\left(\overline{c}_{n}^{2}\right)_{n\in\mathbb{N}}$
of $\left(c_{n}^{2}\right)_{n\in\mathbb{N}}$ and a function $c^{2}\in L^{1}\left(\left[0,2\right],\mathbb{R}\right)$
such that
\[
\overline{c}_{n}^{2}\rightharpoonup c^{2}\mbox{ in }L^{1}\left(\left[0,2\right],\mathbb{R}\right).
\]
Following this pattern we are able to give a recursive definition
of a family

$\left\{ \left(\left(c_{n}^{T}\right)_{n\in\mathbb{N}},\left(\overline{c}_{n}^{T}\right)_{n\in\mathbb{N}},c^{T}\right)/T\in\mathbb{N}\right\} $
and $\left\{ i\left(T,n\right)\in\left[0,+\infty\right)/T,n\in\mathbb{N}\right\} $
satisfying, for every

$T,n\in\mathbb{N}$:
\begin{align}
 & c_{n}^{T}\in\Lambda\left(k_{0}\right),\,\overline{c}_{n}^{T}=c_{n+i\left(T,n\right)}^{T}\nonumber \\
 & U\left(c_{n}^{T+1};k_{0}\right)\geq U\left(\overline{c}_{n}^{T};k_{0}\right)\nonumber \\
 & c_{n}^{T+1}=\overline{c}_{n}^{T}\wedge N\left(k_{0},T+1\right)\mbox{ a.e. in }\left[0,T+1\right]\nonumber \\
 & \overline{c}_{n}^{T}\rightharpoonup c^{T}\mbox{ in }L^{1}\left(\left[0,T\right],\mathbb{R}\right)\label{conv debole barra}
\end{align}
Now we show that, for every $T\in\mathbb{N}$, 
\begin{equation}
c^{T+1}=c^{T}\mbox{ almost everywhere in }\left[0,T\right].\label{gamma buona def}
\end{equation}
Assume the notation ``$\tilde{\forall}s\in A:P\left(s\right)$''
with the meaning `` for almost every $s\in A$, $P\left(s\right)$
holds'' . Hence:
\begin{eqnarray*}
\tilde{\forall}s\in\left[0,T\right]:\overline{c}_{n}^{T+1}\left(s\right) & = & c_{n+i\left(T,n\right)}^{T+1}\left(s\right)\ =\ \overline{c}_{n+i\left(T,n\right)}^{T}\left(s\right)\wedge N\left(k_{0},T+1\right)\ =\ \overline{c}_{n+i\left(T,n\right)}^{T}\left(s\right)
\end{eqnarray*}
where the last equality holds because, by the penultimate condition
in $\eqref{conv debole barra}$ and by the monotonicity of the function
$N\left(k_{0},\cdot\right)$, for any $p\in\mathbb{N}$: 
\[
\left\Vert \overline{c}_{p}^{T}\right\Vert _{\infty,\left[0,T\right]}=\left\Vert c_{p+i\left(T,p\right)}^{T}\right\Vert _{\infty,\left[0,T\right]}\leq N\left(k_{0},T\right)\leq N\left(k_{0},T+1\right).
\]
Hence the assertion in $\eqref{gamma buona def}$ follows from the
essential uniqueness of the weak limit in $L^{1}\left(\left[0,T\right],\mathbb{R}\right)$.

Now we want to set up a diagonal procedure in order to exhibit a sequence
$\left(\gamma_{n}\right)_{n\in\mathbb{N}}\subseteq\Lambda\left(k_{0}\right)$
and a function $\gamma\in\mathcal{L}_{loc}^{1}\left(\left[0,+\infty\right),\mathbb{R}\right)$
such that
\[
\gamma_{n}\rightharpoonup\gamma\mbox{ in }L^{1}\left(\left[0,T\right],\mathbb{R}\right)\quad\forall T>0.
\]

\begin{defn}
\label{definizione quasi ottimale}i) $\gamma:\left[0,+\infty\right)\to\mathbb{R}$
is the function
\[
\gamma\left(t\right):=c^{\left[t\right]+1}\left(t\right)\quad\forall t\geq0
\]
 ii) The sequence $\left(\gamma_{n}\right)_{n\in\mathbb{N}}$ is defined
as follows:
\[
\begin{cases}
\gamma_{1}:=\overline{c}_{1}^{1}\\
\forall n\geq2:\mbox{ if }\gamma_{n}=c_{j\left(n\right)}^{n}\mbox{ then }\gamma_{n+1}=\overline{c}_{m}^{n+1},\\
\mbox{where }m:=\min\left\{ k\in\mathbb{N}/k>j\left(n\right)\mbox{ and }c_{k}^{n+1}\in\left(\overline{c}_{p}^{n+1}\right)_{p\in\mathbb{N}}\right\} 
\end{cases}
\]

\end{defn}
This diagonal procedure is resumed by the following scheme, in which
the elements of the (weakly) convergent subsequences $\left(\overline{c}_{n}^{m}\right)_{n\in\mathbb{N}}$,
$m\geq1$ are emphasized by the square brackets.
\[
\begin{array}{cccccccccccccc}
c_{1}^{1} & c_{2}^{1} & \dots & c_{h}^{1} & \dots & \left[\mathbf{c_{j\left(1\right)}^{1}}\right] & \dots & c_{i}^{1} & \dots & c_{k}^{2} & \dots & c_{j\left(2\right)}^{1} & \dots & c_{j\left(3\right)}^{1}\\
c_{1}^{2} & c_{2}^{2} & \dots & \left[c_{h}^{2}\right] & \dots & c_{j\left(1\right)}^{2} & \dots & c_{i}^{2} & \dots & c_{k}^{2} & \dots & \left[\mathbf{c_{j\left(2\right)}^{2}}\right] & \dots & c_{j\left(3\right)}^{2}\\
c_{1}^{3} & c_{2}^{3} & \dots & c_{h}^{3} & \dots & c_{j\left(1\right)}^{3} & \dots & \left[c_{i}^{3}\right] & \dots & \left[c_{k}^{3}\right] & \dots & c_{j\left(2\right)}^{3} & \dots & \left[\mathbf{c_{j\left(3\right)}^{3}}\right]\\
c_{1}^{4} & \dots & \dots & \dots & \dots & \dots & \dots & \dots & \dots & \dots & \dots & \dots & \dots & \dots
\end{array}
\]

\begin{rem}
\label{conv debole}Let $T\in\mathbb{N}$. Condition $\eqref{gamma buona def}$
implies that $\gamma=c^{T}$ almost everywhere in $\left[0,T\right]$.
Hence it follows from $\eqref{conv debole barra}$ that:
\[
\overline{c}_{n}^{T}\rightharpoonup\gamma\mbox{ in }L^{1}\left(\left[0,T\right],\mathbb{R}\right).
\]

\end{rem}
We have shown that $\left(\overline{c}_{n}^{T+1}\right)_{n\in\mathbb{N}}$,
restricted to $\left[0,T\right]$, almost coincides with a subsequence
of $\left(\overline{c}_{n}^{T}\right)_{n\in\mathbb{N}}$ ; we want
to prove an analogous result in relation to $\left(\gamma_{n}\right)_{n\in\mathbb{N}}$.

We have $\gamma_{1}:=\overline{c}_{1}^{1}=c_{j\left(1\right)}^{1}$
(with $j\left(1\right)=1+i\left(1,1\right))$ , so by Definition $\ref{definizione quasi ottimale}$,
there exists $m_{2}>j\left(1\right)$ such that $\gamma_{2}=\overline{c}_{m_{2}}^{2}=c_{j\left(m_{2}\right)}^{2}$,
where $j\left(m_{2}\right):=m_{2}+i\left(2,m_{2}\right)$. Hence
\begin{eqnarray*}
\tilde{\forall}s\in\left[0,1\right]:\gamma_{2}\left(s\right) & = & \overline{c}_{m_{2}}^{2}\left(s\right)=c_{j\left(m_{2}\right)}^{2}\left(s\right)\ =\ \overline{c}_{j\left(m_{2}\right)}^{1}\left(s\right)\wedge N\left(k_{0},2\right)=\overline{c}_{j\left(m_{2}\right)}^{1}\left(s\right)
\end{eqnarray*}
where the last equality again holds because by construction $\left\Vert \overline{c}_{p}^{1}\right\Vert _{\infty.\left[0,1\right]}\leq N\left(k_{0},1\right)\leq N\left(k_{0},2\right)$
for any $p\in\mathbb{N}$.

Moreover for some $m_{3}>j\left(m_{2}\right)$, $\gamma_{3}=\overline{c}_{m_{3}}^{3}$;
setting $j\left(m_{3}\right):=m_{3}+i\left(3,m_{3}\right)$ and $j\left(j\left(m_{3}\right)\right):=j\left(m_{3}\right)+i\left(2,j\left(m_{3}\right)\right)$,
we have:
\begin{eqnarray*}
\tilde{\forall}s\in\left[0,1\right]:\gamma_{3}\left(s\right) & = & \overline{c}_{m_{3}}^{3}\left(s\right)\ =\ c_{j\left(m_{3}\right)}^{3}\left(s\right)\ =\ \overline{c}_{j\left(m_{3}\right)}^{2}\left(s\right)\wedge N\left(k_{0},3\right)\\
 & = & c_{j\left(j\left(m_{3}\right)\right)}^{2}\left(s\right)\wedge N\left(k_{0},3\right)\ =\ \overline{c}_{j\left(j\left(m_{3}\right)\right)}^{1}\left(s\right)\wedge N\left(k_{0},2\right)\wedge N\left(k_{0},3\right)\\
 & = & \overline{c}_{j\left(j\left(m_{3}\right)\right)}^{1}\left(s\right)
\end{eqnarray*}
as $N\left(k_{0},1\right)\leq N\left(k_{0},2\right)\leq N\left(k_{0},3\right)$,
and
\begin{eqnarray*}
\tilde{\forall}s\in\left[0,2\right]:\gamma_{3}\left(s\right) & = & c_{j\left(m_{3}\right)}^{3}\left(s\right)=\overline{c}_{j\left(m_{3}\right)}^{2}\left(s\right)\wedge N\left(k_{0},3\right)\\
 & = & \overline{c}_{j\left(m_{3}\right)}^{2}
\end{eqnarray*}
 Hence, by the fact that $1<j\left(m_{2}\right)<j\left(j\left(m_{3}\right)\right)$,
we see that $\left(\gamma_{1},\gamma_{2},\gamma_{3}\right)$ coincides
with a subsequence of $\left(\overline{c}_{n}^{1}\right)_{n\in\mathbb{N}}$almost
everywhere in $\left[0,1\right]$; it follws from $j\left(m_{3}\right)>m_{2}$
that $\left(\gamma_{2},\gamma_{3}\right)$ coincides with a subsequence
of $\left(\overline{c}_{n}^{2}\right)_{n\in\mathbb{N}}$ almost everywhere
in $\left[0,2\right]$. Obviously this reasoning can be repeated to
prove by induction the following
\begin{prop}
\label{Step1}Let $\left(\gamma_{n}\right)_{n\in\mathbb{N}}$, $\gamma$
as in Definition $\ref{definizione quasi ottimale}$. Then $\left(\gamma_{n}\right)_{n\in\mathbb{N}}\subseteq\Lambda\left(k_{0}\right)$,$\gamma\in\mathcal{L}_{loc}^{1}\left(\left[0,+\infty\right),\mathbb{R}\right)$
and
\[
\lim_{n\to+\infty}U\left(\gamma_{n};k_{0}\right)=V\left(k_{0}\right).
\]
Moreover, for every $T\in\mathbb{N}$, $\left(\gamma_{n}\right)_{n\geq T}$
coincides almost everywhere in $\left[0,T\right]$ with a subsequence
of $\left(\overline{c}_{n}^{T}\right)_{n\in\mathbb{N}}$. Consequently
\begin{align*}
 & \gamma_{n}\rightharpoonup\gamma\mbox{ in }L^{1}\left(\left[0,T\right],\mathbb{R}\right)\quad\forall T>0,T\in\mathbb{R},\\
 & \left\Vert \gamma_{n}\right\Vert _{\infty,\left[0,T\right]}\leq N\left(k_{0},T\right)\quad\forall T,n\in\mathbb{N}.
\end{align*}
\end{prop}
\begin{proof}
By Remark $\ref{conv debole}$, for every $T\in\mathbb{N}$, $\gamma=c^{T}\mbox{ almost everywhere in }\left[0,T\right]$;
hence $\gamma\in L^{1}\left(\left[0,T\right],\mathbb{R}\right)$,
which implies $\gamma\in\mathcal{L}_{loc}^{1}\left(\left[0,+\infty\right),\mathbb{R}\right)$
because $T$ is generic.

By Definition $\ref{definizione quasi ottimale}$, $\gamma_{1}=c_{j\left(1\right)}^{1}$
for some $j\left(1\right)\geq1$; hence by induction we have, for
every $n\in\mathbb{N}$, $\gamma_{n}=\overline{c}_{j\left(n\right)}^{n}$
for some $j\left(n\right)\geq n$; in particular, by the first condition
in $\eqref{conv debole barra}$, $\gamma_{n}\in\Lambda\left(k_{0}\right)$.
With $n\to j\left(n\right)$ defined this way, set $p\left(n\right):=j\left(n\right)+i\left(n,j\left(n\right)\right)$;
so remembering the other conditions in $\eqref{conv debole barra}$:
\begin{eqnarray*}
\left|U\left(\gamma_{n};k_{0}\right)-V\left(k_{0}\right)\right| & = & V\left(k_{0}\right)-U\left(\gamma_{n};k_{0}\right)=V\left(k_{0}\right)-U\left(\overline{c}_{j\left(n\right)}^{n};k_{0}\right)\\
 & = & V\left(k_{0}\right)-U\left(c_{p\left(n\right)}^{n};k_{0}\right)\leq V\left(k_{0}\right)-U\left(\overline{c}_{p\left(n\right)}^{n-1};k_{0}\right)\\
 & = & V\left(k_{0}\right)-U\left(c_{p\left(n\right)+i\left(n-1,p\left(n\right)\right)}^{n-1};k_{0}\right)\\
 & \leq & \ldots\quad\leq V\left(k_{0}\right)-U\left(c_{q\left(n\right)}^{1};k_{0}\right)\\
 & \leq & V\left(k_{0}\right)-U\left(c_{q\left(n\right)};k_{0}\right)=\left|U\left(c_{q\left(n\right)};k_{0}\right)-V\left(k_{0}\right)\right|,
\end{eqnarray*}
for some $q\left(n\right)\geq p\left(n\right)\geq n$. Hence the first
assertion follows from the fact that $\lim_{k\to+\infty}U\left(c_{k};k_{0}\right)$$=V\left(k_{0}\right)$.

Now fix $T\in\mathbb{N}$. The argument developed after Remark $\ref{conv debole}$
inductively shows that there exists a sequence of natural numbers
$n\to k_{T}\left(n\right)$ such that
\[
\forall n\geq T:\tilde{\forall}s\in\left[0,T\right]:\gamma_{n}\left(s\right)=\overline{c}_{n+k_{T}\left(n\right)}^{T}\left(s\right).
\]
This implies by Remark $\ref{conv debole}$ that $\gamma_{n}\rightharpoonup\gamma\mbox{ in }L^{1}\left(\left[0,T\right],\mathbb{R}\right)$. 

As this holds for every $T\in\mathbb{N}$, it is a consequence of
Remark $\ref{remark weak ovunque}$ that it must hold for every real
number $T>0$. The last condition obviously holds by construction
and by $\eqref{conv debole barra}$.
\end{proof}
The first step is then accomplished. 

$\,$\\
\textbf{\emph{Step 2.}} The next step is to show that $\gamma$ is
admissible at $k_{0}$. For this purpose, it is enough to prove the
following
\begin{prop}
Let $T>0$. Hence $\gamma\geq0$ almost everywhere in $\left[0,T\right]$,
and, for every $t\in\left[0,T\right]$, $k\left(t;k_{0},\gamma\right)\geq0$.\end{prop}
\begin{proof}
It is well known that the weak convergence of $\left(\gamma_{n}\right)_{n\in\mathbb{N}}$
to $\gamma$ in $L^{1}\left(\left[0,T\right],\mathbb{R}\right)$,
ensured by Proposition $\ref{Step1}$, implies that
\[
\liminf_{n\to+\infty}\gamma_{n}\left(t\right)\leq\gamma\left(t\right)\mbox{ a.e. in }\left[0,T\right].
\]
Moreover, $\left(\gamma_{n}\right)_{n\in\mathbb{N}}\subseteq\Lambda\left(k_{0}\right)$,
hence any $\gamma_{n}$ is almost everywhere non-negative in $\left[0,T\right]$.
This implies $\gamma\geq0$ almost everywhere in $\left[0,T\right]$.

Set $\kappa:=k\left(\cdot;k_{0},\gamma\right)$ and $\kappa_{n}:=k\left(\cdot;k_{0},\gamma_{n}\right)$;
we show that, for every $t\in\left[0,T\right]$:
\[
\limsup_{n\to+\infty}\kappa_{n}\left(t\right)\leq\kappa\left(t\right).
\]
Then the second assertion will follow from the fact that $\kappa_{n}\geq0$
in $\left[0,T\right]$ for any $n\in\mathbb{N}$, by the admissibility
of the $\gamma_{n}$'s.

Fix $n\in\mathbb{N}$. Subtracting the state equation for $\kappa$
from the state equation for $\kappa_{n}$, we obtain, for every $t\in\left[0,T\right]$:
\begin{eqnarray*}
\dot{\kappa_{n}}\left(t\right)-\dot{\kappa}\left(t\right) & = & F\left(\kappa_{n}\left(t\right)\right)-F\left(\kappa\left(t\right)\right)-\left[\gamma_{n}\left(t\right)-\gamma\left(t\right)\right]\ \leq\ \overline{M}\left[\kappa_{n}\left(t\right)-\kappa\left(t\right)\right]-\left[\gamma_{n}\left(t\right)-\gamma\left(t\right)\right]
\end{eqnarray*}
which implies
\[
\left[\dot{\kappa_{n}}\left(t\right)-\dot{\kappa}\left(t\right)\right]e^{-\overline{M}t}-e^{-\overline{M}t}\overline{M}\left[\kappa_{n}\left(t\right)-\kappa\left(t\right)\right]\leq e^{-\overline{M}t}\left[\gamma\left(t\right)-\gamma_{n}\left(t\right)\right]
\]
that is to say:
\begin{eqnarray*}
\frac{\mbox{d}}{\mbox{d}t}\left[\left[\kappa_{n}\left(t\right)-\kappa\left(t\right)\right]e^{-\overline{M}t}\right] & \leq & e^{-\overline{M}t}\left[\gamma\left(t\right)-\gamma_{n}\left(t\right)\right].
\end{eqnarray*}
Hence, for every fixed $t\in\left[0,T\right]$:
\begin{eqnarray*}
\kappa_{n}\left(t\right)-\kappa\left(t\right) & \leq & \int_{0}^{t}e^{\overline{M}\left(t-s\right)}\left[\gamma\left(s\right)-\gamma_{n}\left(s\right)\right]\mbox{d}s\ =\ \int_{0}^{T}\chi_{\left[0,t\right]}\left(s\right)e^{\overline{M}\left(t-s\right)}\left[\gamma\left(s\right)-\gamma_{n}\left(s\right)\right]\mbox{d}s.
\end{eqnarray*}
The function $s\to\chi_{\left[0,t\right]}\left(s\right)e^{\overline{M}\left(t-s\right)}$
is bounded in $\left[0,T\right]$ (by $1$ and $e^{\overline{M}t}$),
hence we can apply the weak convergence $\gamma_{n}\rightharpoonup\gamma$
in $L^{1}\left(\left[0,T\right],\mathbb{R}\right)$ to deduce that
the quantity at the right-hand member of the above inequality tends
to $0$ as $n\to+\infty$. Hence

\[
\limsup_{n\to+\infty}\kappa_{n}\left(t\right)\leq\kappa\left(t\right).
\]

\end{proof}
As a consequence, $\gamma$ is almost everywhere non-negative in $\left[0,+\infty\right)$
and

$k\left(\cdot;k_{0},\gamma\right)$ is everywhere non-negative in
$\left[0,+\infty\right)$ - which precisely means that $\gamma\in\Lambda\left(k_{0}\right)$.
Hence the second step is also ended.

$\,$\\
\textbf{\emph{Step 3.}} Now it is time to define the control which
is optimal at $k_{0}$. In order to do this, we need to extract a
subsequence from $\left(\gamma_{n}\right)_{n\in\mathbb{N}}$ because
the weak convergence to $\gamma$ in the intervals could not be enough
to ensure that $\lim_{n\to+\infty}U\left(\gamma_{n};k_{0}\right)=U\left(\gamma;k_{0}\right)$;
we will also need the admissibility of $\gamma$. By the last assertion
stated in Proposition $\ref{Step1}$, and by the monotonicity of $u$,
we have:
\[
\left\Vert u\left(\gamma_{n}\right)\right\Vert _{\infty,\left[0,1\right]}\leq u\left(N\left(k_{0},1\right)\right)\quad\forall n\in\mathbb{N}.
\]
 Hence by Lemma $\ref{Dunf-Pett}$, there exists a function $f^{1}\in L^{1}\left(\left[0,1\right],\mathbb{R}\right)$
and a sequence $\left(u\left(\gamma_{1,n}\right)\right)_{n\in\mathbb{N}}$
extracted from $\left(u\left(\gamma_{n}\right)\right)_{n\in\mathbb{N}}$,
such that
\[
u\left(\gamma_{1,n}\right)\rightharpoonup f^{1}\mbox{ in }L^{1}\left(\left[0,1\right],\mathbb{R}\right).
\]
Again by Proposition $\ref{Step1}$ and the monotonicity of $u$,
\[
\left\Vert u\left(\gamma_{1,n}\right)\right\Vert _{\infty,\left[0,2\right]}\leq u\left(N\left(k_{0},2\right)\right)\quad\forall n\in\mathbb{N}
\]
which implies by Lemma $\ref{Dunf-Pett}$ the existence of $f^{2}\in L^{1}\left(\left[0,2\right],\mathbb{R}\right)$
and of a sequence $\left(u\left(\gamma_{2,n}\right)\right)_{n\in\mathbb{N}}$
extracted from $\left(u\left(\gamma_{1,n}\right)\right)_{n\in\mathbb{N}}$
such that
\[
u\left(\gamma_{2,n}\right)\rightharpoonup f^{2}\mbox{ in }L^{1}\left(\left[0,2\right],\mathbb{R}\right);
\]
in particular $f^{2}=f^{1}$ almost everywhere in $\left[0,1\right]$
by the essential uniqueness of the weak limit.

Going on this way we see that there exists a family $\left\{ \left(u\left(\gamma_{T,n}\right)_{n\in\mathbb{N}},f^{T}\right)/T\in\mathbb{N}\right\} $
satisfying, for every $T\in\mathbb{N}$:
\begin{align*}
 & \left\Vert u\left(\gamma_{T,n}\right)\right\Vert _{\infty,\left[0,T\right]}\leq u\left(N\left(k_{0},T\right)\right)\quad\forall n\in\mathbb{N}\\
 & \left(u\left(\gamma_{T+1,n}\right)\right)_{n\in\mathbb{N}}\mbox{ is extracted from }\left(u\left(\gamma_{T,n}\right)\right)_{n\in\mathbb{N}}\\
 & f^{T+1}=f^{T}\mbox{ almost everywhere in }\left[0,T\right]\\
 & u\left(\gamma_{T,n}\right)\rightharpoonup f^{T}\mbox{ in }L^{1}\left(\left[0,T\right],\mathbb{R}\right).
\end{align*}
Hence, for every $T\in\mathbb{N}$, the sequence $\left(u\left(\gamma_{n,n}\right)\right)_{n\geq T}$
is extracted from $\left(u\left(\gamma_{T,n}\right)\right)_{n\in\mathbb{N}}$
. If we define $f\left(t\right):=f^{\left[t\right]+1}\left(t\right)$,
then $f=f^{T}$ almost everywhere in $\left[0,T\right]$. So
\begin{equation}
u\left(\gamma_{n,n}\right)\rightharpoonup f\mbox{ in }L^{1}\left(\left[0,T\right],\mathbb{R}\right)\quad\forall T>0.\label{conv debole u(gamma_n,n)}
\end{equation}

by construction and by Remark $\ref{remark weak ovunque}$. This implies
that
\[
0\leq\liminf_{n\to+\infty}u\left(\gamma_{n,n}\left(t\right)\right)\leq f\left(t\right)
\]
for almost every $t\in\mathbb{R}$.

Now define $c^{*}:\left[0,+\infty\right)\to\mathbb{R}$ as 
\[
c^{*}\left(t\right):=\begin{cases}
u^{-1}\left(f\left(t\right)\right) & \mbox{ if }f\left(t\right)\geq0\\
0 & \mbox{ if }f\left(t\right)<0.
\end{cases}
\]
Obviously $c^{*}\geq0$ everywhere in $\mathbb{R}$. Moreover, again
by the properties of the weak convergence, for any $T\in\mathbb{N}$
and for almost every $t\in\left[0,T\right]$:
\[
f\left(t\right)\leq\limsup_{n\to+\infty}u\left(\gamma_{n,n}\left(t\right)\right)\leq u\left(N\left(k_{0},T\right)\right).
\]
This implies, together with the fact that $u^{-1}$ is increasing,
that $c^{*}$ is bounded above by $N\left(k_{0},T\right)$ almost
everywhere in $\left[0,T\right]$. As this holds for every $T\in\mathbb{N}$,
\begin{equation}
c^{*}\in L_{loc}^{\infty}\left(\left[0,+\infty\right),\mathbb{R}\right).\label{c* L1 loc}
\end{equation}
To complete the proof of the admissibility of $c^{*}$, we show that
$c^{*}\leq\gamma$ almost everywhere in $\left[0,+\infty\right)$.

Fix $T>0$ and let $t_{0}\in\left[0,T\right]$ be a Lebesgue point
for both $f$ and $\gamma$ in $\left[0,T\right]$; then take $t_{1}\in\left(t_{0},T\right)$.
By the concavity of $u$ and by Jensen inequality:
\begin{equation}
\frac{\int_{t_{0}}^{t_{1}}u\left(\gamma_{n,n}\left(s\right)\right)\mbox{d}s}{t_{1}-t_{0}}\leq u\left(\frac{\int_{t_{0}}^{t_{1}}\gamma_{n,n}\left(s\right)\mbox{d}s}{t_{1}-t_{0}}\right)\label{per c*<gamma}
\end{equation}

Observe that $\left(\gamma_{n,n}\right)_{n\geq1}$ is a subsequence
of $\left(\gamma_{1,n}\right)_{n\in\mathbb{N}}$, which is in its
turn extracted from $\left(\gamma_{n}\right)_{n\in\mathbb{N}}$. Hence
$\gamma_{n,n}\rightharpoonup\gamma\mbox{ in }L^{1}\left(\left[0,T\right],\mathbb{R}\right)$,
which implies $\lim_{n\to+\infty}\int_{t_{0}}^{t_{1}}\gamma_{n,n}\left(s\right)\mbox{d}s=\int_{t_{0}}^{t_{1}}\gamma\left(s\right)\mbox{d}s.$
So taking the limit for $n\to+\infty$ in $\eqref{per c*<gamma}$,
by the continuity of $u$ and by $\eqref{conv debole u(gamma_n,n)}$,
we have:
\[
\frac{\int_{t_{0}}^{t_{1}}f\left(s\right)\mbox{d}s}{t_{1}-t_{0}}\leq u\left(\frac{\int_{t_{0}}^{t_{1}}\gamma\left(s\right)\mbox{d}s}{t_{1}-t_{0}}\right).
\]
As $t_{0}$ is a Lebesgue point for both $f$ and $\gamma$ in $\left[0,T\right]$,
we can take the limit for $t_{1}\to t_{0}$ in the previous inequality
and get $f\left(t_{0}\right)\leq u\left(\gamma\left(t_{0}\right)\right)$.

By the Lebesgue Point Theorem, this argument works for almost every
$t_{0}\in\left[0,T\right]$. So by the monotonicity of $u^{-1}$ we
deduce
\[
c^{*}\leq\gamma\mbox{ almost everywhere in }\left[0,T\right].
\]
Because $T$ is generic, we have by $\eqref{eq: comp ODE weak}$:
$k\left(t;k_{0},c^{*}\right)\geq k\left(t;k_{0},\gamma\right)$ for
every $t\in\mathbb{R}$. Hence by the admissibility of $\gamma$ at
$k_{0}$, $k\left(\cdot;k_{0},c^{*}\right)\geq0$. This implies, together
with $\eqref{c* L1 loc}$ and $c^{*}\geq0$ in $\left[0,+\infty\right)$,
\[
c^{*}\in\Lambda\left(k_{0}\right).
\]
Then by Proposition $\ref{Step1}$, by the fact that $\left(\gamma_{n,n}\right)_{n\in\mathbb{N}}$
is extracted from $\left(\gamma_{n}\right)_{n\in\mathbb{N}}$, by
Lemma $\ref{funzionale finito}$, iii) , by $\eqref{conv debole u(gamma_n,n)}$
and by Fatou's Lemma:
\begin{eqnarray*}
V\left(k_{0}\right) & = & \lim_{n\to+\infty}U\left(\gamma_{n};k_{0}\right)=\lim_{n\to+\infty}U\left(\gamma_{n,n};k_{0}\right)\\
 & = & \lim_{n\to+\infty}\rho\int_{0}^{+\infty}e^{-\rho t}\int_{0}^{t}u\left(\gamma_{n,n}\left(s\right)\right)\mbox{ds}\mbox{d}t\\
 & \leq & \rho\int_{0}^{+\infty}e^{-\rho t}\limsup_{n\to+\infty}\int_{0}^{t}u\left(\gamma_{n,n}\left(s\right)\right)\mbox{ds}\mbox{d}t\\
 & = & \rho\int_{0}^{+\infty}e^{-\rho t}\int_{0}^{t}f\left(s\right)\mbox{ds}\mbox{d}t\\
 & = & \rho\int_{0}^{+\infty}e^{-\rho t}\int_{0}^{t}u\left(c^{*}\left(s\right)\right)\mbox{ds}\mbox{d}t=U\left(c^{*};k_{0}\right).
\end{eqnarray*}
Hence we have proved that for every $k_{0}\geq0$ there exists $c^{*}\in\Lambda\left(k_{0}\right)$
which is optimal at $k_{0}$ and everywhere positive in $\mathbb{R}$,
satisfying:
\[
c^{*}\in L_{loc}^{\infty}\left(\left[0,+\infty\right),\mathbb{R}\right).
\]

\section{Further properties of the value function}

Now it is possible to set some regularity properties of the value
function, with the help of optimal controls. The next theorem uses
the monotonicity with respect to the first variable of the function
defined in Lemma $\ref{controlli limitati e meglio}$.
\begin{thm}
\label{Further prop V}The value function $V:\left[0,+\infty\right)\to\mathbb{R}$
satisfies:

i) V is strictly increasing

ii) For every $k_{0}>0$, there exists $C\left(k_{0}\right),\delta>0$
such that for every $h\in\left(-\delta,\delta\right)$:
\[
\frac{V\left(k_{0}+h\right)-V\left(k_{0}\right)}{h}\geq C\left(k_{0}\right)
\]

iii) V is Lipschitz-continuous in every closed sub-interval of $\left(0,+\infty\right)$.\end{thm}
\begin{proof}
i) Let $0<k_{0}<k_{1}$. Set $c\in\left(0,F\left(k_{0}\right)\right]$
and $c_{0}\equiv c$ in $\left[0,+\infty\right)$; hence by Lemma
$\ref{caratt controlli ammiss costanti}$ and by Theorem $\ref{Basic Prop V}$,
\[
V\left(0\right)=0<\frac{u\left(c\right)}{\rho}=U\left(c_{0};k_{0}\right)\leq V\left(k_{0}\right).
\]
In order to establish that $V\left(k_{0}\right)<V\left(k_{1}\right)$,
take $c\in\Lambda\left(k_{0}\right)$ optimal at $k_{0}$ and define
$\underline{c}^{k_{1}-k_{0}}$ as in Lemma $\ref{lemma per monotonia V}$.
As
\[
u'\left(N\left(k_{0},k_{1}-k_{0}\right)+1\right)\int_{0}^{k_{1}-k_{0}}e^{-\rho t}\mbox{d}t>0
\]
we have
\[
V\left(k_{0}\right)=U\left(c;k_{0}\right)<U\left(\underline{c}^{k_{1}-k_{0}};k_{1}\right)\leq V\left(k_{1}\right)
\]

ii) We split the proof in two parts.

First, take $k_{0},h>0$, $c$ optimal at $k_{0}$ and set $k_{1}:=k_{0}+h$.
Because $k_{1}>k_{0}$ we can choose $\underline{c}^{k_{1}-k_{0}}=\underline{c}^{h}\in\Lambda\left(k_{0}+h\right)$
as in Lemma $\ref{lemma per monotonia V}$. Hence
\begin{eqnarray*}
V\left(k_{0}+h\right)-V\left(k_{0}\right) & \geq & U\left(\underline{c}^{h};k_{0}+h\right)-U\left(c;k_{0}\right)\ \geq\ u'\left(N\left(k_{0},h\right)+1\right)\int_{0}^{h}e^{-\rho t}\mbox{d}t
\end{eqnarray*}
Now, by the fact that $\lim_{h\to0}\frac{1}{h}\int_{0}^{h}e^{-\rho t}\mbox{d}t=1$
and that $N\left(k_{0},\cdot\right)$ is increasing, there exists
$\delta>0$ such that, for any $h\in\left(0,\delta\right)$: 
\begin{eqnarray*}
\frac{V\left(k_{0}+h\right)-V\left(k_{0}\right)}{h} & \geq & u'\left(N\left(k_{0},h\right)+1\right)\frac{\int_{0}^{h}e^{-\rho t}\mbox{d}t}{h}\ \geq\ \frac{u'\left(N\left(k_{0},1\right)+1\right)}{2}=:C\left(k_{0}\right)
\end{eqnarray*}

In the second place, fix $k_{0}>0$, $h<0$ and $c$ optimal at $k_{0}+h$.

Then again take $\underline{c}^{k_{0}-\left(k_{0}+h\right)}=\underline{c}^{-h}\in\Lambda\left(k_{0}\right)$
as in Lemma $\ref{lemma per monotonia V}$. Hence
\begin{eqnarray*}
V\left(k_{0}+h\right)-V\left(k_{0}\right) & \leq & U\left(c;k_{0}+h\right)-U\left(\underline{c}^{-h};k_{0}\right)\\
 & \leq & -u'\left(N\left(k_{0}+h,-h\right)+1\right)\int_{0}^{-h}e^{-\rho t}\mbox{d}t.
\end{eqnarray*}
We can assume that $-\frac{1}{h}\int_{0}^{-h}e^{-\rho t}\mbox{d}t\geq\frac{1}{2}$
for $-\delta<h<0$. Hence, by the monotonicity of $N\left(\cdot,\cdot\right)$
in both variables, for every $h\in\left(-\delta,0\right)$:
\[
\frac{V\left(k_{0}+h\right)-V\left(k_{0}\right)}{h}\geq\frac{u'\left(N\left(k_{0}+h,-h\right)+1\right)}{2}\geq\frac{u'\left(N\left(k_{0},1\right)+1\right)}{2}=C\left(k_{0}\right).
\]

iii) Let $0<k_{0}<k_{1}$. We want a reverse inequality for $V\left(k_{1}\right)-V\left(k_{0}\right)$,
so take $c_{1}\in\Lambda\left(k_{1}\right)$ optimal at $k_{1}$.
In order to define the proper $c_{0}\in\Lambda\left(k_{0}\right)$,
observe that the orbit $k=k\left(\cdot;k_{0},0\right)$ (with null
control) satisfies $\dot{k}=F\left(k\right)$. With an argument similar
to the one used in Proposition $\ref{caratt controlli ammiss costanti}$
we can see that $\dot{k}\left(t\right)>F\left(k_{0}\right)>0$ for
every $t>0$, and so $\lim_{t\to+\infty}k\left(t\right)=+\infty$.

Then by Darboux's property there exists $\bar{t}>0$ such that $k\left(\bar{t}\right)=k_{1}$.
Observe that, since $k$ and $F$ are strictly increasing functions,
$\dot{k}$ must also be strictly increasing. Hence appling Lagrange's
thorem to $k$ gives for some $\xi\in\left(0,\bar{t}\right)$:
\begin{eqnarray}
k_{1}-k_{0} & = & k\left(\bar{t}\right)-k\left(0\right)=\bar{t}\cdot\dot{k}\left(\xi\right)>\bar{t}\dot{k}\left(0\right)=\bar{t}F\left(k_{0}\right)\label{eq:lip}
\end{eqnarray}
Now define
\[
c_{0}\left(t\right):=\begin{cases}
0 & \mbox{ if }t\in\left[0,\bar{t}\right]\\
c_{1}\left(t-\bar{t}\right) & \mbox{ if }t>\bar{t}
\end{cases}
\]
 It is easy to check that $c_{0}\in\Lambda\left(k_{0}\right)$, because
\begin{align*}
 & k\left(t;k_{0},c_{0}\right)=k\left(t;k_{0},0\right)>0\quad\forall t\in\left[0,\bar{t}\right]\\
 & k\left(t+\bar{t};k_{0},c_{0}\right)=k\left(t;k_{1},c_{1}\right)\geq0\quad\forall t\geq0
\end{align*}
 by the uniqueness of the orbit; as far as the second equality is
concerned, observe that both orbits pass through $\left(0,k_{1}\right)$
and satisfy the differential equation controlled with $c_{1}$ for
$t>0$. Hence by $\eqref{eq:lip}$:
\begin{eqnarray*}
V\left(k_{1}\right)-V\left(k_{0}\right) & \leq & U\left(c_{1};k_{1}\right)-U\left(c_{0};k_{0}\right)=\int_{0}^{+\infty}e^{-\rho t}\left[u\left(c_{1}\left(t\right)\right)-u\left(c_{0}\left(t\right)\right)\right]\mbox{d}t\\
 & = & \int_{0}^{+\infty}e^{-\rho t}u\left(c_{1}\left(t\right)\right)\mbox{d}t-\int_{\bar{t}}^{+\infty}e^{-\rho t}u\left(c_{1}\left(t-\bar{t}\right)\right)\mbox{d}t\\
 & = & \int_{0}^{+\infty}e^{-\rho t}u\left(c_{1}\left(t\right)\right)\mbox{d}t-\int_{0}^{+\infty}e^{-\rho\left(s+\bar{t}\right)}u\left(c_{1}\left(s\right)\right)\mbox{d}s\\
 & = & \left(1-e^{-\rho\bar{t}}\right)U\left(c_{1};k_{1}\right)\,=\,\left(1-e^{-\rho\bar{t}}\right)V\left(k_{1}\right)\,\leq\,\rho\bar{t}V\left(k_{1}\right)\,<\,\rho V\left(k_{1}\right)\frac{k_{1}-k_{0}}{F\left(k_{0}\right)}
\end{eqnarray*}
So by the monotonicity of $V$ and $F$ we have, for $a\leq k_{0}<k_{1}\leq b$:
\[
V\left(k_{1}\right)-V\left(k_{0}\right)\leq\rho\frac{V\left(b\right)}{F\left(a\right)}\left(k_{1}-k_{0}\right).
\]

\end{proof}

\section{Dynamic Programming}

In this section we study the properties of the value function as a
solution to Bellman and Hamilton-Jacobi-Bellman equations.

First observe that we can translate an orbit by translating the control,
according to the next remark.
\begin{rem}[Translation of the orbit]
\label{translation orbit}For every $k_{0}\geq0$ and every $c\in\mathcal{L}_{loc}^{1}\mbox{\ensuremath{\left(\left(0,+\infty\right),\mathbb{R}\right)}}$:
\[
k\left(\cdot;k\left(\tau;k_{0},c\right),c\left(\cdot+\tau\right)\right)=k\left(\cdot+\tau;k_{0},c\right)
\]
by the uniqueness of the orbit. In particular, if $c\in\Lambda\left(k_{0}\right)$
then $c\left(\cdot+\tau\right)\in\Lambda\left(k\left(\tau;k_{0},c\right)\right)$.
\end{rem}
$\,$

The first step consists in proving a suitable version of Dynamic Programming
Principle.

$\,$
\begin{thm}[\textbf{Bellman's Dynamic Programming Principle}]
\label{BE-GF}For every $\tau>0$, the value function $V:[0,+\infty)\to\mathbb{R}$
satisfies the following functional equation:
\begin{equation}
\forall k_{0}\geq0:\mathbf{\mathrm{v}}\left(k_{0}\right)=\sup_{c\in\Lambda\left(k_{0}\right)}\left\{ \int_{0}^{\tau}e^{-\rho t}u\left(c\left(t\right)\right)\mbox{d}t+e^{-\rho\tau}\mathrm{v}\left(k\left(\tau;k_{0},c\right)\right)\right\} \label{eq:BE-GF}
\end{equation}

in the unknown $\mathrm{v}:[0,+\infty)\to\mathbb{R}$.\end{thm}
\begin{proof}
Fix $\tau>0$ and $k_{0}\geq0$, and set
\[
\sigma\left(\tau,k_{0}\right):=\sup_{c\in\Lambda\left(k_{0}\right)}\left\{ \int_{0}^{\tau}e^{-\rho t}u\left(c\left(t\right)\right)\mbox{d}t+e^{-\rho\tau}V\left(k\left(\tau;k_{0},c\right)\right)\right\} .
\]
We prove that
\[
\sigma\left(\tau,k_{0}\right)=\sup_{c\in\Lambda\left(k_{0}\right)}U\left(c;k_{0}\right).
\]
In the first place, we show that $\sigma\left(\tau,k_{0}\right)$
is an upper bound of $\left\{ U\left(c;k_{0}\right)\,/\, c\in\Lambda\left(k_{0}\right)\right\} $.

Fix $c\in\Lambda\left(k_{0}\right)$; then by Remark $\ref{translation orbit}$
$c\left(\cdot+\tau\right)\in\Lambda\left(k\left(\tau;k_{0},c\right)\right)$;
hence
\begin{eqnarray*}
\sigma\left(\tau,k_{0}\right) & \geq & \int_{0}^{\tau}e^{-\rho t}u\left(c\left(t\right)\right)\mbox{d}t+e^{-\rho\tau}V\left(k\left(\tau;k_{0},c\right)\right)\\
 & \geq & \int_{0}^{\tau}e^{-\rho t}u\left(c\left(t\right)\right)\mbox{d}t+e^{-\rho\tau}U\left(c\left(\cdot+\tau\right);k\left(\tau;k_{0},c\right)\right)\\
 & = & \int_{0}^{\tau}e^{-\rho t}u\left(c\left(t\right)\right)\mbox{d}t+\int_{0}^{+\infty}e^{-\rho\left(t+\tau\right)}u\left(c\left(t+\tau\right)\right)\mbox{d}t\\
 & = & \int_{0}^{\tau}e^{-\rho t}u\left(c\left(t\right)\right)\mbox{d}t+\int_{\tau}^{+\infty}e^{-\rho s}u\left(c\left(s\right)\right)\mbox{d}t=U\left(c;k_{0}\right)
\end{eqnarray*}
In the second place, fix $\epsilon>0$, and take 
\[
0<\epsilon'\leq\frac{2\epsilon}{\left(1+e^{-\rho\tau}\right)}.
\]
Hence there exists $\tilde{c}_{\epsilon}\in\Lambda\left(k_{0}\right)$
and $\tilde{\tilde{c}}_{\epsilon}\in\Lambda\left(k\left(\tau;k_{0},\tilde{c}_{\epsilon}\right)\right)$
such that
\begin{eqnarray*}
\sigma\left(\tau,k_{0}\right)-\epsilon & \leq & \sigma\left(\tau,k_{0}\right)-\frac{\epsilon'}{2}\left(1+e^{-\rho\tau}\right)\\
 & \leq & \int_{0}^{\tau}e^{-\rho t}u\left(\tilde{c}_{\epsilon}\left(t\right)\right)\mbox{d}t+e^{-\rho\tau}V\left(k\left(\tau;k_{0},\tilde{c}_{\epsilon}\right)\right)-e^{-\rho\tau}\frac{\epsilon'}{2}\\
 & \leq & \int_{0}^{\tau}e^{-\rho t}u\left(\tilde{c}_{\epsilon}\left(t\right)\right)\mbox{d}t+e^{-\rho\tau}U\left(\tilde{\tilde{c}}_{\epsilon};k\left(\tau;k_{0},\tilde{c}_{\epsilon}\right)\right)\\
 & = & \int_{0}^{\tau}e^{-\rho t}u\left(\tilde{c}_{\epsilon}\left(t\right)\right)\mbox{d}t+\int_{0}^{+\infty}e^{-\rho\left(t+\tau\right)}u\left(\tilde{\tilde{c}}_{\epsilon}\left(t\right)\right)\mbox{d}t
\end{eqnarray*}
Now set 
\[
c_{\epsilon}\left(t\right):=\begin{cases}
\tilde{c}_{\epsilon}\left(t\right) & \mbox{ if }t\in\left[0,\tau\right]\\
\tilde{\tilde{c}}_{\epsilon}\left(t-\tau\right) & \mbox{ if }t>\tau
\end{cases}
\]
Hence $c_{\epsilon}\in\mathcal{L}_{loc}^{1}\left(\left(0,+\infty\right),\mathbb{R}\right)$
and $\forall t>0:c_{\epsilon}\left(t+\tau\right)=\tilde{\tilde{c}}_{\epsilon}\left(t\right)$.
So:
\begin{eqnarray}
\sigma\left(\tau,k_{0}\right)-\epsilon & \leq & \int_{0}^{+\infty}e^{-\rho t}u\left(c_{\epsilon}\left(t\right)\right)\mbox{d}t\label{serve per BE}
\end{eqnarray}
Finally, it is easy to show that $c_{\epsilon}\in\Lambda\left(k_{0}\right)$.
Observe that $k\left(\cdot;k_{0},c_{\epsilon}\right)=k\left(\cdot;k_{0},\tilde{c}_{\epsilon}\right)$
in $\left[0,\tau\right]$ by definition of $c_{\epsilon}$ and by
uniqueness. In particular $k\left(\tau;k_{0},c_{\epsilon}\right)=k\left(\tau;k_{0},\tilde{c}_{\epsilon}\right)$,
so that $k\left(\cdot+\tau;k_{0},c_{\epsilon}\right)$ and $k\left(\cdot;k\left(\tau;k_{0},\tilde{c}_{\epsilon}\right),\tilde{\tilde{c}}_{\epsilon}\right)$
have the same initial value; moreover, these two orbits satisfy the
same state equation (i.e. the equation associated to the control $c_{\epsilon}\left(\cdot+\tau\right)$)
and so they coincide, again by uniqueness. Recalling that by definition
$\tilde{c}_{\epsilon}\in\Lambda\left(k_{0}\right)$ and $\tilde{\tilde{c}}_{\epsilon}\in\Lambda\left(k\left(\tau;k_{0},\tilde{c}_{\epsilon}\right)\right)$,
we have $k\left(t;k_{0},c_{\epsilon}\right)\geq0$ for all $t\geq0$.
Hence by $\eqref{serve per BE}$ we can write
\[
\sigma\left(\tau,k_{0}\right)-\epsilon\leq U\left(c_{\epsilon};k_{0}\right)
\]
 and the assertion is proven.
\end{proof}
Equation $\eqref{eq:BE-GF}$ is called \emph{Bellman Functional Equation}.

A consequence of the above theorem is that every control which is
optimal respect to a state, is also optimal respect to every following
optimal state.
\begin{cor}
Let $k_{0}\geq0$, $c^{*}\in\Lambda\left(k_{0}\right)$ . Hence the
following are equivalent:

i) $c^{*}$ is optimal at $k_{0}$

ii) For every $\tau>0$:
\[
V\left(k_{0}\right)=\int_{0}^{\tau}e^{-\rho t}u\left(c^{*}\left(t\right)\right)\mbox{d}t+e^{-\rho\tau}V\left(k\left(\tau;k_{0},c^{*}\right)\right)
\]
Moreover, i) or ii) imply that for every $\tau>0$, $c^{*}\left(\cdot+\tau\right)$
is admissible and optimal at $k\left(\tau;k_{0},c^{*}\right)$.\end{cor}
\begin{proof}
i) $\Rightarrow$ ii) Let us assume that $c^{*}$ is admissible and
optimal at $k_{0}\geq0$ and fix $\tau>0$. Observe that $c^{*}\left(\cdot+\tau\right)$
is admissible at $k\left(\tau;k_{0},c^{*}\right)$ by Remark $\ref{translation orbit}$.
Hence, by Theorem $\ref{BE-GF}$:
\begin{eqnarray}
V\left(k_{0}\right) & \geq & \int_{0}^{\tau}e^{-\rho t}u\left(c^{*}\left(t\right)\right)\mbox{d}t+e^{-\rho\tau}V\left(k\left(\tau;k_{0},c^{*}\right)\right)\nonumber \\
 & \geq & \int_{0}^{\tau}e^{-\rho t}u\left(c^{*}\left(t\right)\right)\mbox{d}t+e^{-\rho\tau}U\left(c^{*}\left(\cdot+\tau\right);k\left(\tau;k_{0},c^{*}\right)\right)\nonumber \\
 & = & \int_{0}^{+\infty}e^{-\rho t}u\left(c^{*}\left(t\right)\right)\mbox{d}t\,=\, U\left(c^{*};k_{0}\right)\,=\, V\left(k_{0}\right).\label{coroll be aggiunta}
\end{eqnarray}
Hence
\begin{equation}
V\left(k_{0}\right)=\int_{0}^{\tau}e^{-\rho t}u\left(c^{*}\left(t\right)\right)\mbox{d}t+e^{-\rho\tau}V\left(k\left(\tau;k_{0},c^{*}\right)\right).\label{caratt. ottimale}
\end{equation}
ii) $\Rightarrow$ i) Suppose that $c^{*}\in\Lambda\left(k_{0}\right)$
and $\eqref{caratt. ottimale}$ holds for every $\tau>0$. For every
$\epsilon>0$ pick $\hat{c}_{\epsilon}\in\Lambda\left(k\left(\frac{1}{\epsilon};k_{0},c^{*}\right)\right)$
such that:
\begin{equation}
V\left(k\left(\frac{1}{\epsilon};k_{0},c^{*}\right)\right)-\epsilon\leq U\left(\hat{c}_{\epsilon};k\left(\frac{1}{\epsilon};k_{0},c^{*}\right)\right).\label{optimal aggiunta}
\end{equation}
Then define
\[
c_{\epsilon}\left(t\right):=\begin{cases}
c^{*}\left(t\right) & \mbox{if }t\in\left[0,\frac{1}{\epsilon}\right]\\
\hat{c}_{\epsilon}\left(t-\frac{1}{\epsilon}\right) & \mbox{if }t>\frac{1}{\epsilon}
\end{cases}
\]
By the same arguments we used in the proof of Theorem $\ref{BE-GF}$
, $c_{\epsilon}\in\Lambda\left(k_{0}\right)$ and, obviously, $c_{\epsilon}\left(t+\frac{1}{\epsilon}\right)=\hat{c}_{\epsilon}\left(t\right)$
for every $t>0$.

Hence, taking $\tau=1/\epsilon$ in $\eqref{caratt. ottimale}$, we
have by $\eqref{optimal aggiunta}$:
\begin{eqnarray}
V\left(k_{0}\right)-\epsilon e^{-\rho/\epsilon} & = & \int_{0}^{1/\epsilon}e^{-\rho t}u\left(c^{*}\left(t\right)\right)\mbox{d}t+e^{-\rho/\epsilon}\left[V\left(k\left(\frac{1}{\epsilon};k_{0},c^{*}\right)\right)-\epsilon\right]\nonumber \\
 & \leq & \int_{0}^{1/\epsilon}e^{-\rho t}u\left(c^{*}\left(t\right)\right)\mbox{d}t+e^{-\rho/\epsilon}U\left(\hat{c}_{\epsilon};k\left(\frac{1}{\epsilon};k_{0},c^{*}\right)\right)\nonumber \\
 & = & \int_{0}^{1/\epsilon}e^{-\rho t}u\left(c^{*}\left(t\right)\right)\mbox{d}t+\int_{0}^{+\infty}e^{-\rho\left(t+\frac{1}{\epsilon}\right)}u\left(c_{\epsilon}\left(t+\frac{1}{\epsilon}\right)\right)\mbox{d}t\nonumber \\
 & = & \int_{0}^{1/\epsilon}e^{-\rho t}u\left(c^{*}\left(t\right)\right)\mbox{d}t+\int_{1/\epsilon}^{+\infty}e^{-\rho s}u\left(c_{\epsilon}\left(s\right)\right)\mbox{d}s\label{stima V -epsilon_per_exp}
\end{eqnarray}
Now we show that the second addend tends to $0$ as $\epsilon\to0$
.Observe that by Jensen inequality, for every $T\geq1/\epsilon$:
\begin{eqnarray}
\int_{1/\epsilon}^{T}e^{-\rho s}u\left(c_{\epsilon}\left(s\right)\right)\mbox{d}s & = & \left[e^{-\rho s}\int_{1/\epsilon}^{s}u\left(c_{\epsilon}\left(\tau\right)\right)\mbox{d}\tau\right]_{s=1/\epsilon}^{s=T}+\rho\int_{1/\epsilon}^{T}e^{-\rho s}\int_{1/\epsilon}^{s}u\left(c_{\epsilon}\left(\tau\right)\right)\mbox{d}\tau\mbox{d}s\nonumber \\
 & \leq & e^{-\rho T}\int_{0}^{T}u\left(c_{\epsilon}\left(\tau\right)\right)\mbox{d}\tau+\rho\int_{1/\epsilon}^{T}e^{-\rho s}\int_{0}^{s}u\left(c_{\epsilon}\left(\tau\right)\right)\mbox{d}\tau\mbox{d}s\nonumber \\
 & \leq & e^{-\rho T}\int_{0}^{T}u\left(c_{\epsilon}\left(\tau\right)\right)\mbox{d}\tau+\rho\int_{1/\epsilon}^{T}se^{-\rho s}u\left(\frac{\int_{0}^{s}c_{\epsilon}\left(\tau\right)\mbox{d}\tau}{s}\right)\mbox{d}s\nonumber \\
 & \to & \rho\int_{1/\epsilon}^{+\infty}se^{-\rho s}u\left(\frac{\int_{0}^{s}c_{\epsilon}\left(\tau\right)\mbox{d}\tau}{s}\right)\mbox{d}s\mbox{\quad\ensuremath{\mbox{as }T\to+\infty},}\label{stima su Ic_epsilon}
\end{eqnarray}

by Lemma $\ref{funzionale finito}$, ii) and by the admissibility
of $c_{\epsilon}$. By point i) of the same Lemma, for every $\epsilon<1$
and every $s\geq1/\epsilon$:
\begin{eqnarray*}
se^{-\rho s}u\left(\frac{\int_{0}^{s}c_{\epsilon}\left(\tau\right)\mbox{d}\tau}{s}\right) & \leq & se^{-\rho s}u\left(M\left(k_{0}\right)\left[1+e^{\left(L+\epsilon_{0}\right)s}\right]+\frac{M\left(k_{0}\right)}{s\left(L+\epsilon_{0}\right)}\right)\\
 & \leq & se^{-\rho s}\Biggl\{ u\left(M\left(k_{0}\right)\right)+M\left(k_{0}\right)u\left(e^{\left(L+\epsilon_{0}\right)s}\right)+u\left(\frac{M\left(k_{0}\right)}{L+\epsilon_{0}}\right)\Biggl\}
\end{eqnarray*}

which implies, together with $\eqref{stima su Ic_epsilon}$, for every
$\epsilon<1$:
\begin{eqnarray*}
0\leq\int_{1/\epsilon}^{+\infty}e^{-\rho s}u\left(c_{\epsilon}\left(s\right)\right)\mbox{d}s & \leq & \rho\int_{1/\epsilon}^{+\infty}se^{-\rho s}u\left(\frac{\int_{0}^{s}c_{\epsilon}\left(\tau\right)\mbox{d}\tau}{s}\right)\mbox{d}s\\
 & \leq & \rho\left[u\left(M\left(k_{0}\right)\right)+u\left(\frac{M\left(k_{0}\right)}{L+\epsilon_{0}}\right)\right]\int_{1/\epsilon}^{+\infty}se^{-\rho s}\mbox{d}s+\\
 & + & \rho M\left(k_{0}\right)\int_{1/\epsilon}^{+\infty}se^{-\rho s}u\left(e^{\left(L+\epsilon_{0}\right)s}\right)\mbox{d}s.
\end{eqnarray*}
By Remark $\ref{remark 1}$ this quantity tends to $0$ as $\epsilon\to0$. 

Hence, letting $\epsilon\to0$ in $\eqref{stima V -epsilon_per_exp}$,
we find:
\[
V\left(k_{0}\right)\leq\int_{0}^{+\infty}e^{-\rho t}u\left(c^{*}\left(t\right)\right)\mbox{d}t=U\left(c^{*};k_{0}\right)
\]
which implies that $c^{*}$ is optimal at $k_{0}$.

Finally, if i) holds, then by $\eqref{coroll be aggiunta}$: 

\[
V\left(k\left(\tau;k_{0},c^{*}\right)\right)=U\left(c^{*}\left(\cdot+\tau\right);k\left(\tau;k_{0},c^{*}\right)\right).
\]

\end{proof}
A careful study of the difference quotients for the functions
\[
t\to e^{-\rho t}V\left(k\left(t\right)\right)
\]
(for an orbit $k$) leads to the following definitions and theorems.

$\,$
\begin{defn}
\label{Def: C+}Let $f\in\mathcal{C}^{0}\left(\left(0,+\infty\right),\mathbb{R}\right)$;
we say that $f\in\mathcal{C}^{+}\left(\left(0,+\infty\right),\mathbb{R}\right)$
if, and only if, for every $k_{0}>0$ there exist $\delta,C^{+},C^{-}>0$
such that
\begin{align*}
 & \frac{f\left(k_{0}+h\right)-f\left(k_{0}\right)}{h}\geq C^{+}\quad\forall h\in\left(0,\delta\right)\\
 & \frac{f\left(k_{0}+h\right)-f\left(k_{0}\right)}{h}\geq C^{-}\quad\forall h\in\left(-\delta,0\right)
\end{align*}

\end{defn}
We note that by Theorem $\ref{Further prop V}$, (ii) the value function
$V$ satisfies
\begin{equation}
V\in\mathcal{C}^{+}\left(\left(0,+\infty\right),\mathbb{R}\right).\label{V e' C+}
\end{equation}

\begin{defn}
\label{Def Hamiltoniana}The function $H:\left[0,+\infty\right)\times\left(0,+\infty\right)\to\mathbb{R}$
defined by
\[
H\left(k,p\right):=-\sup\left\{ \left[F\left(k\right)-c\right]\cdot p+u\left(c\right)\,/\, c\in\left[0,+\infty\right)\right\} 
\]
is called \emph{Hamiltonian}. 

The equation
\begin{equation}
\rho\mathrm{v}\left(k\right)+H\left(k,\mathrm{v}'\left(k\right)\right)=0\quad\forall k>0\label{eq:HJB-GF}
\end{equation}
in the unknown $\mathrm{v}\in\mathcal{C}^{+}\left(\left(0,+\infty\right),\mathbb{R}\right)\cap\mathcal{C}^{1}\left(\left(0,+\infty\right),\mathbb{R}\right)$
is called \emph{Hamilton-Jacobi-Bellman equation} (HJB).
\end{defn}
Observe that any solution of $\eqref{eq:HJB-GF}$ must be strictly
increasing, by Definition $\ref{Def: C+}$.
\begin{rem}
The Hamiltonian is always finite. Indeed
\[
-\sup_{c\in\left[0,+\infty\right)}\left\{ \left[F\left(k\right)-c\right]\cdot p+u\left(c\right)\right\} >-\infty\iff p>0.
\]
If $p>0$, since $\lim_{c\to+\infty}u'\left(c\right)=0$ we can choose
$c_{p}\geq0$ such that $u'\left(c_{p}\right)\leq p$; this implies
by the concavity of $u$:
\[
\forall c\geq0:u\left(c\right)-cp\leq u\left(c\right)-u'\left(c_{p}\right)c\leq u\left(c_{p}\right)-u'\left(c_{p}\right)c_{p},
\]
so that
\[
-F\left(k\right)p-\sup_{c\in\left[0,+\infty\right)}\left\{ u\left(c\right)-cp\right\} \geq-F\left(k\right)p-u\left(c_{p}\right)+u'\left(c_{p}\right)c_{p}>-\infty.
\]
Otherwise, when $p\leq0$ , since $\lim_{c\to+\infty}u\left(c\right)=+\infty$
we have
\[
-F\left(k\right)p-\sup_{c\in\left[0,+\infty\right)}\left\{ u\left(c\right)-cp\right\} \leq-F\left(k\right)p-\sup_{c\in\left[0,+\infty\right)}u\left(c\right)=-\infty.
\]
\end{rem}
\begin{defn}
A function $v\in\mathcal{C}^{+}\left(\left(0,+\infty\right),\mathbb{R}\right)$
is called a \emph{viscosity subsolution }{[}\emph{supersolution}{]}\emph{
of} (HJB) if, and only if:

for every $\varphi\in\mathcal{C}^{1}\left(\left(0,+\infty\right),\mathbb{R}\right)$
and for every local maximum {[}minimum{]} point $k_{0}>0$ of $v-\varphi$:
\begin{eqnarray*}
\rho v\left(k_{0}\right)-\sup\left\{ \left[F\left(k_{0}\right)-c\right]\cdot\varphi'\left(k_{0}\right)+u\left(c\right)\,/\, c\in\left[0,+\infty\right)\right\}  & =\\
\rho v\left(k_{0}\right)+H\left(k_{0},\varphi'\left(k_{0}\right)\right) & \leq & 0\\
 & [\geq & 0]
\end{eqnarray*}

If $v$ is both a viscosity subsolution of (HJB) and a viscosity supersolution
of (HJB), then we say that $v$ is a\emph{ viscosity solution of}
(HJB).\end{defn}
\begin{rem}
\label{remark V puo essere VSS}The latter definition is well posed.
Indeed, let $v\in\mathcal{C}^{+}\left(\left(0,+\infty\right),\mathbb{R}\right)$
and $\varphi\in\mathcal{C}^{1}\left(\left(0,+\infty\right),\mathbb{R}\right)$.
If $k_{0}$ is a local maximum for $v-\varphi$ in $\left(0,+\infty\right)$,
then for $h<0$ big enough we have:
\begin{align*}
 & v\left(k_{0}\right)-v\left(k_{0}+h\right)\geq\varphi\left(k_{0}\right)-\varphi\left(k_{0}+h\right)\implies\\
 & 0<C^{-}\leq\frac{v\left(k_{0}\right)-v\left(k_{0}+h\right)}{h}\leq\frac{\varphi\left(k_{0}\right)-\varphi\left(k_{0}+h\right)}{h}.
\end{align*}
If $k_{0}$ is a local minimum for $v-\varphi$ in $\left(0,+\infty\right)$,
then for $h>0$ small enough we have:
\begin{align*}
 & v\left(k_{0}\right)-v\left(k_{0}+h\right)\leq\varphi\left(k_{0}\right)-\varphi\left(k_{0}+h\right)\implies\\
 & 0<C^{+}\leq\frac{v\left(k_{0}\right)-v\left(k_{0}+h\right)}{h}\leq\frac{\varphi\left(k_{0}\right)-\varphi\left(k_{0}+h\right)}{h}.
\end{align*}
In both cases, we have $\varphi'\left(k_{0}\right)>0$.
\end{rem}
We are now going to prove that the value function is a viscosity solution
of (HJB). As pointed out in the introduction, this will be done without
any regularity assumption on $H$; nevertheless, this function can
be easily shown to be continuous, since for every $k\geq0$, $p>0$:
\[
H\left(k,p\right)=F\left(k\right)p+\left(-u\right)^{*}\left(p\right),
\]
where $\left(-u\right)^{*}$ is the (convex) conjugate function of
the convex function $-u$.
\begin{lem}
\label{seq-orbit GF-1}Let $k_{0}>0$ and $\left(c_{T}\right)_{T>0}\subseteq\Lambda\left(k_{0}\right)$
satisfying:
\[
\left\Vert c_{T}\right\Vert _{\infty,\left[0,T\right]}\leq N\left(k_{0},T\right)\quad\forall T>0.
\]
where $N$ is the function defined in \emph{Lemma} $\ref{controlli limitati e meglio}$.
Hence
\[
\forall T\in\left[0,1\right]:\forall t\in\left[0,T\right]:\left|k\left(t;k_{0},c_{T}\right)-k_{0}\right|\leq Te^{\bar{M}t}\left[F\left(k_{0}\right)+N\left(k_{0},1\right)\right].
\]
In particular $k\left(T;k_{0},c_{T}\right)\to k_{0}$ as $T\to0$.\end{lem}
\begin{proof}
Set $k_{0}$ and $\left(c_{T}\right)_{T>0}$ as in the hypothesis
and fix $0\leq T\leq1$. Hence integrating both sides of the state
equation we get, for every $t\in\left[0,T\right]$:
\begin{eqnarray*}
k\left(t;k_{0},c_{T}\right)-k_{0} & = & \int_{0}^{t}\left[F\left(k_{0}\right)-c_{T}\left(s\right)\right]\mbox{d}s+\int_{0}^{t}\left[F\left(k\left(s;k_{0},c_{T}\right)\right)-F\left(k_{0}\right)\right]\mbox{d}s
\end{eqnarray*}
which implies by Remark $\ref{F Lip}$:
\begin{eqnarray*}
\left|k\left(t;k_{0},c_{T}\right)-k_{0}\right| & \leq & \int_{0}^{t}\left|F\left(k_{0}\right)-c_{T}\left(s\right)\right|\mbox{d}s+\int_{0}^{t}\left|F\left(k\left(s;k_{0},c_{T}\right)\right)-F\left(k_{0}\right)\right|\mbox{d}s\\
 & \leq & \int_{0}^{T}\left|F\left(k_{0}\right)-c_{T}\left(s\right)\right|\mbox{d}s+\bar{M}\int_{0}^{t}\left|k\left(s;k_{0},c_{T}\right)-k_{0}\right|\mbox{d}s
\end{eqnarray*}
Hence by Gronwall's inequality and by the monotonicity of $N\left(k_{0},\cdot\right)$,
for every $T\in\left[0,1\right]$ and every $t\in\left[0,T\right]$:
\begin{eqnarray*}
\left|k\left(t;k_{0},c_{T}\right)-k_{0}\right| & \leq & e^{\bar{M}t}\int_{0}^{T}\left|F\left(k_{0}\right)-c_{T}\left(s\right)\right|\mbox{d}s.\\
 & \leq & Te^{\bar{M}t}\left[F\left(k_{0}\right)+N\left(k_{0},T\right)\right]\\
 & \leq & Te^{\bar{M}t}\left[F\left(k_{0}\right)+N\left(k_{0},1\right)\right].
\end{eqnarray*}
\end{proof}
\begin{prop}
The value function $V:\left[0,+\infty\right)\to\mathbb{R}$ is a viscosity
solution of \emph{(HJB)}.

Consequently, if $V\in\mathcal{C}^{1}\left(\left[0,+\infty\right),\mathbb{R}\right)$,
then $V$ is strictly increasing and is a solution of \emph{(HJB)
- $\eqref{eq:HJB-GF}$ }in the classical sense.\end{prop}
\begin{proof}
In the first place, we show that $V$ is a viscosity supersolution
of (HJB).

Let $\varphi\in\mathcal{C}^{1}\left(\left(0,+\infty\right),\mathbb{R}\right)$
and $k_{0}>0$ be a local minumum point of $V-\varphi$, so that
\begin{equation}
V\left(k_{0}\right)-V\leq\varphi\left(k_{0}\right)-\varphi\label{eq: supersol 1}
\end{equation}
in a proper neighbourhood of $k_{0}$. Now fix $c\in\left[0,+\infty\right)$
and set $k:=k\left(\cdot;k_{0},c\right)$. As $k_{0}>0$, there exists
$T_{c}>0$ such that $k>0$ in $\left[0,T_{c}\right]$. Hence the
control
\[
\tilde{c}\left(t\right):=\begin{cases}
c & \mbox{ if }t\in\left[0,T_{c}\right]\\
0 & \mbox{ if }t>T_{c}
\end{cases}
\]
is admissible at $k_{0}$. Then by Theorem $\ref{BE-GF}$, for every
$\tau\in\left[0,T_{c}\right]$:
\begin{eqnarray*}
V\left(k_{0}\right)-V\left(k\left(\tau\right)\right) & \geq & \int_{0}^{\tau}e^{-\rho t}u\left(\tilde{c}\left(t\right)\right)\mbox{d}t+V\left(k\left(\tau\right)\right)\left[e^{-\rho\tau}-1\right]\\
 & = & u\left(c\right)\int_{0}^{\tau}e^{-\rho t}\mbox{d}t+V\left(k\left(\tau\right)\right)\left[e^{-\rho\tau}-1\right].
\end{eqnarray*}
Hence by $\eqref{eq: supersol 1}$ and by the continuity of $k$,
we have for every $\tau>0$ sufficiently small: 
\begin{eqnarray*}
\frac{\varphi\left(k\left(0\right)\right)-\varphi\left(k\left(\tau\right)\right)}{\tau} & \geq & u\left(c\right)\frac{\int_{0}^{\tau}e^{-\rho t}\mbox{d}t}{\tau}+V\left(k\left(\tau\right)\right)\frac{\left[e^{-\rho\tau}-1\right]}{\tau}.
\end{eqnarray*}
Letting $\tau\to0$ and using the continuity of $V$ and $k$:
\[
-\varphi'\left(k_{0}\right)\left[F\left(k_{0}\right)-c\right]\geq u\left(c\right)-\rho V\left(k_{0}\right)
\]
 which implies, taking the sup for $c\geq0$:
\[
\rho V\left(k_{0}\right)+H\left(k_{0},\varphi'\left(k_{0}\right)\right)\geq0
\]

Secondly we show that $V$ is a viscosity subsolution of (HJB).

Let $\varphi\in\mathcal{C}^{1}\left(\left(0,+\infty\right),\mathbb{R}\right)$
and $k_{0}>0$ be a local maximum point of $V-\varphi$, so that
\begin{equation}
V\left(k_{0}\right)-V\geq\varphi\left(k_{0}\right)-\varphi\label{v subsol max loc}
\end{equation}
in a proper neighborhood $\mathcal{N}\left(k_{0}\right)$ of $k_{0}$.

Fix $\epsilon>0$ and, using the definition of $V$, define a family
of controls $\left(c_{T,\epsilon}\right)_{T>0}\subseteq\Lambda\left(k_{0}\right)$
such that for every $T>0$:
\begin{equation}
V\left(k_{0}\right)-T\epsilon\leq U\left(c_{T,\epsilon};k_{0}\right).\label{fighetto}
\end{equation}
Now take $\left(c_{T,\epsilon}\right)^{T}$ as in Lemma $\ref{controlli limitati e meglio}$
and set $\bar{c}_{T,\epsilon}:=\left(c_{T,\epsilon}\right)^{T}$ for
simplicity of notation (so that $\bar{c}_{T,\epsilon}\in\Lambda\left(k_{0}\right)$).
We have:
\begin{eqnarray*}
V\left(k_{0}\right)-T\epsilon & \leq & U\left(c_{T,\epsilon};k_{0}\right)\leq U\left(\bar{c}_{T,\epsilon};k_{0}\right)\\
 & = & \int_{0}^{T}e^{-\rho t}u\left(\bar{c}_{T,\epsilon}\left(t\right)\right)\mbox{d}t+e^{-\rho T}\int_{T}^{+\infty}e^{-\rho\left(s-T\right)}u\left(\bar{c}_{T,\epsilon}\left(s-T+T\right)\right)\mbox{d}s\\
 & = & \int_{0}^{T}e^{-\rho t}u\left(\bar{c}_{T,\epsilon}\left(t\right)\right)\mbox{d}t+e^{-\rho T}U\left(\bar{c}_{T,\epsilon}\left(\cdot+T\right);k\left(T;k_{0},\bar{c}_{T,\epsilon}\right)\right)\\
 & \leq & \int_{0}^{T}e^{-\rho t}u\left(\bar{c}_{T,\epsilon}\left(t\right)\right)\mbox{d}t+e^{-\rho T}V\left(k\left(T;k_{0},\bar{c}_{T,\epsilon}\right)\right)
\end{eqnarray*}
where we have used Remark $\ref{translation orbit}$.

By Lemma $\ref{seq-orbit GF-1}$ we have for $T>0$ sufficiently small
(say $T<\hat{T}$), 
\[
k\left(T;k_{0},\bar{c}_{T,\epsilon}\right)\in\mathcal{N}\left(k_{0}\right).
\]

Hence, setting $\bar{k}_{T,\epsilon}:=k\left(\cdot;k_{0},\bar{c}_{T,\epsilon}\right)$,
for every $T<\hat{T}$, we have by $\eqref{v subsol max loc}$:
\begin{eqnarray*}
\varphi\left(k_{0}\right)-\varphi\left(\bar{k}_{T,\epsilon}\left(T\right)\right)-e^{-\rho T}V\left(\bar{k}_{T,\epsilon}\left(T\right)\right) & \leq & V\left(k_{0}\right)-V\left(\bar{k}_{T,\epsilon}\left(T\right)\right)-e^{-\rho T}V\left(\bar{k}_{T,\epsilon}\left(T\right)\right)\\
 & \leq & \int_{0}^{T}e^{-\rho t}u\left(\bar{c}_{T,\epsilon}\left(t\right)\right)\mbox{d}t-V\left(\bar{k}_{T,\epsilon}\left(T\right)\right)+T\epsilon
\end{eqnarray*}
which implies
\begin{eqnarray}
 &  & \int_{0}^{T}-\left\{ \varphi'\left(\bar{k}_{T,\epsilon}\left(t\right)\right)\left[F\left(\bar{k}_{T,\epsilon}\left(t\right)\right)-\bar{c}_{T,\epsilon}\left(t\right)\right]+e^{-\rho t}u\left(\bar{c}_{T,\epsilon}\left(t\right)\right)\right\} \mbox{d}t\nonumber \\
 &  & \leq\, V\left(\bar{k}_{T,\epsilon}\left(T\right)\right)\left[e^{-\rho T}-1\right]+T\epsilon.\label{ineq: subsol-parziale}
\end{eqnarray}
Observe that the integral at the left hand member bigger than: 
\begin{eqnarray}
 &  & \int_{0}^{T}-\left\{ \left[\varphi'\left(k_{0}\right)+\omega_{1}\left(t\right)\right]\left[F\left(k_{0}\right)-\bar{c}_{T,\epsilon}\left(t\right)+\omega_{2}\left(t\right)\right]+u\left(\bar{c}_{T,\epsilon}\left(t\right)\right)\right\} \mbox{d}t=\nonumber \\
 &  & \int_{0}^{T}-\left\{ \varphi'\left(k_{0}\right)\left[F\left(k_{0}\right)-\bar{c}_{T,\epsilon}\left(t\right)\right]+u\left(\bar{c}_{T,\epsilon}\left(t\right)\right)\right\} \mbox{d}t+\nonumber \\
 &  & +\int_{0}^{T}-\left\{ \varphi'\left(k_{0}\right)\omega_{2}\left(t\right)\mbox{d}t+\omega_{1}\left(t\right)\left[\omega_{2}\left(t\right)+F\left(k_{0}\right)-\bar{c}_{T,\epsilon}\left(t\right)\right]\right\} \mbox{d}t\label{subsol-parziale 2}
\end{eqnarray}
where $\omega_{1}$, $\omega_{2}$ are functions which are continuous
in a neighborhood of $0$ and satisfy:
\[
\omega_{1}\left(0\right)=\omega_{2}\left(0\right)=0.
\]
This implies, for $T<1$:
\begin{align*}
 & \left|\int_{0}^{T}\varphi'\left(k_{0}\right)\omega_{2}\left(t\right)\mbox{d}t+\int_{0}^{T}\omega_{1}\left(t\right)\left[\omega_{2}\left(t\right)+F\left(k_{0}\right)-\bar{c}_{T,\epsilon}\left(t\right)\right]\mbox{d}t\right|\\
 & \leq\,\left|\varphi'\left(k_{0}\right)\right|o_{1}\left(T\right)+o_{2}\left(T\right)+\int_{0}^{T}\left|\omega_{1}\left(t\right)\right|\left[F\left(k_{0}\right)+\bar{c}_{T,\epsilon}\left(t\right)\right]\mbox{d}t\\
 & \leq\,\left|\varphi'\left(k_{0}\right)\right|o_{1}\left(T\right)+o_{2}\left(T\right)+\left[F\left(k_{0}\right)+N\left(k_{0},T\right)\right]o_{3}\left(T\right)\\
 & \leq\,\left|\varphi'\left(k_{0}\right)\right|o_{1}\left(T\right)+o_{2}\left(T\right)+\left[F\left(k_{0}\right)+N\left(k_{0},1\right)\right]o_{3}\left(T\right)
\end{align*}
where
\[
\lim_{T\to0}\frac{o_{i}\left(T\right)}{T}=0
\]
for $i=1,2,3$. Observe that this is true even if the $o_{i}$s depend
on $T$, by Lemma $\ref{seq-orbit GF-1}$. For instance,
\begin{eqnarray*}
\left|o_{1}\left(T\right)\right| & = & \left|\int_{0}^{T}\omega_{2}\left(t\right)\mbox{d}t\right|\leq T\max_{\left[0,T\right]}\left|\omega_{2}\right|=T\left|\omega_{2}\left(\tau_{T}\right)\right|\\
 & = & T\left|F\left(\bar{k}_{T,\epsilon}\left(\tau_{T}\right)\right)-F\left(k_{0}\right)\right|\\
 & \leq & \overline{M}T\left|\bar{k}_{T,\epsilon}\left(\tau_{T}\right)-k_{0}\right|\leq\overline{M}T^{2}e^{\bar{M}\tau_{T}}\left[F\left(k_{0}\right)+N\left(k_{0},1\right)\right]
\end{eqnarray*}
Moreover, by the fact that $V\in\mathcal{C}^{+}\left(\left[0,+\infty\right),\mathbb{R}\right)$
and by Remark $\ref{remark V puo essere VSS}$, we have for any $t\in\left[0,T\right]$:
\begin{eqnarray*}
-\left\{ \varphi'\left(k_{0}\right)\left[F\left(k_{0}\right)-\bar{c}_{T,\epsilon}\left(t\right)\right]+u\left(\bar{c}_{T,\epsilon}\left(t\right)\right)\right\}  & \geq & -\sup_{c\geq0}\left\{ \varphi'\left(k_{0}\right)\left[F\left(k_{0}\right)-c\right]+u\left(c\right)\right\} \\
 & = & H\left(k_{0},\varphi'\left(k_{0}\right)\right)>-\infty,
\end{eqnarray*}
by which we can write:
\begin{align*}
 & \int_{0}^{T}-\left\{ \varphi'\left(k_{0}\right)\left[F\left(k_{0}\right)-\bar{c}_{T,\epsilon}\left(t\right)\right]+u\left(\bar{c}_{T,\epsilon}\left(t\right)\right)\right\} \mbox{d}t\,\geq\, T\cdot H\left(k_{0},\varphi'\left(k_{0}\right)\right).
\end{align*}
Hence, by $\eqref{ineq: subsol-parziale}$ and $\eqref{subsol-parziale 2}$:
\begin{align*}
 & V\left(\bar{k}_{T,\epsilon}\left(T\right)\right)\left[e^{-\rho T}-1\right]+T\epsilon\,\\
 & \geq\,-\int_{0}^{T}\left\{ \varphi'\left(k_{0}\right)\left[F\left(k_{0}\right)-\bar{c}_{T,\epsilon}\left(t\right)\right]+u\left(\bar{c}_{T,\epsilon}\left(t\right)\right)\right\} \mbox{d}t+\\
 & \ +\int_{0}^{T}-\left\{ \varphi'\left(k_{0}\right)\omega_{2}\left(t\right)\mbox{d}t+\omega_{1}\left(t\right)\left[\omega_{2}\left(t\right)+F\left(k_{0}\right)-\bar{c}_{T,\epsilon}\left(t\right)\right]\mbox{d}t\right\} \\
 & \geq\, T\cdot H\left(k_{0},\varphi'\left(k_{0}\right)\right)+o_{T\to0}\left(T\right)
\end{align*}
for any $0<T<1,\hat{T}$. Hence dividing by $T$, and then letting
$T\to0$, again by Lemma $\ref{seq-orbit GF-1}$ and the continuity
of $V$ we obtain:
\[
-\rho V\left(k_{0}\right)+\epsilon\geq H\left(k_{0},\varphi'\left(k_{0}\right)\right)
\]
which proves the assertion since $\epsilon$ is arbitrary.\end{proof}


\begin{thebibliography}{10}
\bibitem{Aske}Askenazy, P., and Le Van, C. (1999). \textsl{A Model
of Optimal Growth Strategy}, Journal of Economic Theory, 85(1), 27-54.

\bibitem{Barro}Barro, R.J. and X. Sala-i-Martin (1999). \textsl{Economic
Growth} (MIT Press, London).

\bibitem{Cesari}Cesari L. (1983): \textsl{Optimization - Theory and
Applications} (Springer-Verlag, New York).

\bibitem{Edwards}Edwards, R. E. (1995). \textsl{Functional Analysis}
(Holt, Rineheart and Winston, New York)

\bibitem{Gozzi Fiaschi}Gozzi F. and Fiaschi D. (2009). \textsl{Endogenous
growth with convexo-concave technology} (draft). 

\bibitem{Lucas}Lucas, R.E. (1988). \textsl{On the Mechanics of Economic
Development}, Journal of Monetary Economics, 22, 3-42. 

\bibitem{Ramsey}Ramsey, F. P. (1928). \textsl{A Mathematical Theory
of Saving}, The Economic Journal, 38(152), 543-559. 

\bibitem{Romer}Romer, P. M. (1986). \textsl{Increasing Return and Long-Run
Growth}, Journal of Political Economy, 94(5), 1002-1035.

\bibitem{Skiba}Skiba, A. K. (1978). \textsl{Optimal Growth with a Convex-Concave
Production Function}, Econometrica, 46(3), 527-539.

\bibitem{Zhou}Yong, J, Zhou, X. (1999): \textsl{Stochastic Controls
- Hamiltonian Systems and HJB Equations} (Springer-Verlag, New-York
1999).

\bibitem[11]{Zabczyck}Zabczyk, J (1995): \textsl{Mathematical Control
Theory - An Introduction} (Birkhäuser, Boston).\end{thebibliography}
\end{document}